\documentclass[11pt,a4paper]{article}
\usepackage[english]{babel}
\usepackage{amsfonts}
\usepackage[utf8]{inputenc}
\usepackage{amsthm}
\usepackage{ mathrsfs }
\usepackage{ amsmath }
\usepackage{mathrsfs}
\usepackage{enumerate}
\usepackage{cite}
\usepackage{bbm}
\usepackage{makeidx}
\usepackage{graphicx}
\usepackage{frontespizio}
\usepackage{color}
\usepackage[a4paper, total={6.7in, 9in}]{geometry}
\usepackage[makeroom]{cancel}
\usepackage[normalem]{ulem}
\usepackage{ amssymb }
\usepackage{hyperref}
\usepackage{mathtools}
\usepackage{enumitem}
\usepackage[toc,page]{appendix}
\usepackage{ esint }
\usepackage{float}
\usepackage[margin=1cm]{caption}

\author{Ivan Yuri Violo\footnote{\href{mailto:ivan.y.violo@jyu.fi}{ivan.y.violo@jyu.fi}, Department of Mathematics and Statistics, P.O.\ Box 35 (MaD), FI-40014 University of Jyvaskyla, Finland.}}

\newcommand{\diam}{\rm{diam}}

\newtheorem{theorem}{Theorem}[section]
\newtheorem{lemma}[theorem]{Lemma}
\newtheorem{prop}[theorem]{Proposition}
 
\newcommand{\restr}[1]{\lower3pt\hbox{$|_{#1}$}}
\newcommand{\la}{\langle}
\newcommand{\ra}{\rangle}
\newcommand{\sfd}{{\sf d}}

\newcommand{\nn}{\mathbb{N}}

\theoremstyle{definition}
\newtheorem{definition}[theorem]{Definition}
\newtheorem{question}[theorem]{Question}
\newtheorem{open}[theorem]{Open Problem}
\renewcommand{\phi}{\varphi}

\newtheorem{remark}[theorem]{Remark}

\newtheorem{cor}[theorem]{Corollary}

\newcommand{\nchi}{{\raise.3ex\hbox{$\chi$}}}

\DeclareMathOperator\rr{\mathbb{R}}
\DeclareMathOperator\eps{\varepsilon}

\makeindex
\numberwithin{equation}{section}

\renewcommand{\hat}{\widehat}

\newcommand{\T}{{\sf T}}
\newcommand{\arc}{{\rm arc}}
\newcommand{\Int}{{\rm int}}
\newcommand{\cl}{{\rm cl}}
\newcounter{numquote}

\title{Fat triangles inscribed in arbitrary planar domains}
\date{}
\begin{document}
\maketitle

\begin{abstract}
    In 1964 A.\ Bruckner observed that any bounded open set in the plane has an inscribed triangle, that is a triangle contained in the open set and with the vertices lying on the boundary. We prove that this triangle can be taken uniformly  fat, more precisely having all internal angles larger than $\sim 0.3$ degrees, independently of the choice of the initial   open set.  

     We also build a polygon in which all the inscribed triangles are not too-fat, meaning that at least one angle is less than $\sim 55$ degrees. 

   These results show the existence of a  maximal number $\Theta$ strictly between 0 and 60, whose exact value remains unknown, for which all bounded open sets admit an inscribed triangle with all angles larger than or equal to $\Theta$ degrees. 
\end{abstract}


\section{Introduction}
The goal of this note is, broadly speaking, to investigate how fat  the triangles inscribed in a given open set in the plane can be.   We say that a closed triangle $\T\subset \rr^2$ is {\emph{inscribed} in an open set $\Omega\subset \rr^2$ if the vertices of $\T$ belong to $\partial \Omega$ and if the rest of $\T$, together with its interior, lies in  $\Omega$. We stress that, if the interior of one side of $\T$ intersects $\partial \Omega$, then $\T$ is \emph{not} inscribed in $\Omega.$ Note also that we are not making any convexity assumption on $\Omega.$  It was observed by A.\ Bruckner in 1964 that inscribed triangles exist  in  every bounded open set.
\begin{theorem}[{Bruckner, \cite{bruckner}}]\label{thm:bruck}
    Given any non-empty  bounded open set $\Omega\subset \rr^2$ there exists a  non-degenerate triangle inscribed in $\Omega.$
\end{theorem}
It was also noted in \cite{bruckner} that the same result does not hold if we replace triangles with polygons with more than three sides (without self-intersections). An immediate counterexample is given by an open set whose boundary consists of three strictly convex arcs. Boundedness is also necessary, as can be seen by taking  $\Omega=\{(x,y)\in \rr^2 \ : \ y<x^2\}.$ 

In view of Theorem \ref{thm:bruck}, we  can now ask what are the shapes and sizes of the triangles inscribed in a given open set.  Even if very basic, to the best of our knowledge, this question has not been much investigated. We mention however that the topic of inserting polygons inside regions of the plane satisfying suitable constraints is widely studied (see e.g.\ \cite{conv1,conv2,conv3,conv4,conv5,tria1,tria2,tria3}) and has interest also in computational geometry (see for example \cite{comp1,comp2,comp3,comp4}).

We consider first the special case in which $\Omega$ is the interior of  a planar Jordan curve. Recall that given any  $\gamma$  Jordan curve (i.e.\ a simple closed  curve in the plane) the open set $\rr^2\setminus \gamma$ consists of two open connected components, one bounded called the interior  and one unbounded called the exterior.  
A first natural question is the following: 
\begin{question}\label{q:equilateral}
Given any Jordan curve $\gamma$, is there an \emph{equilateral triangle} $\T$ inscribed in the interior of $\gamma$?
\end{question}
Let us first recall that, if we relax the hypotheses on $\T$ requiring only that its vertices lie in $\gamma$, 
allowing the interior of $\T$ to intersect the exterior of $\gamma$, the answer to this question is affirmative and is a  classical result  proved in \cite{meyerson} (see also \cite{nielsen}). Even more in \cite{meyerson} it is shown that any Jordan curve $\gamma$ contains three points that are the vertices of a triangle similar to any given one. The set of such points is actually dense in $\gamma$ (see \cite{nielsen}).
We mention that a  related  famous old open problem  is to determine whether any Jordan curve contains four points that are the vertices of a square. There is a vast literature about this and related questions (see e.g.\ \cite{peg,aspratio,bruckner,cyclic,neck,root3,rectangles,trich,mats,notet,tao} and also \cite{survey1,survey2} for surveys on the topic). 

Nevertheless  we show that Question \ref{q:equilateral} has negative answer, even for polygons.
\begin{theorem}\label{thm:no equilaterals}
    There exists a simple closed polygonal curve with no equilateral triangles inscribed in its interior. 
\end{theorem} 
The polygonal curve $\gamma$ of Theorem \ref{thm:no fat} is explicit and shown in Figure \ref{fig:intro}.
\begin{figure}[!ht]
\centering
\includegraphics[width=10cm, height=3cm]{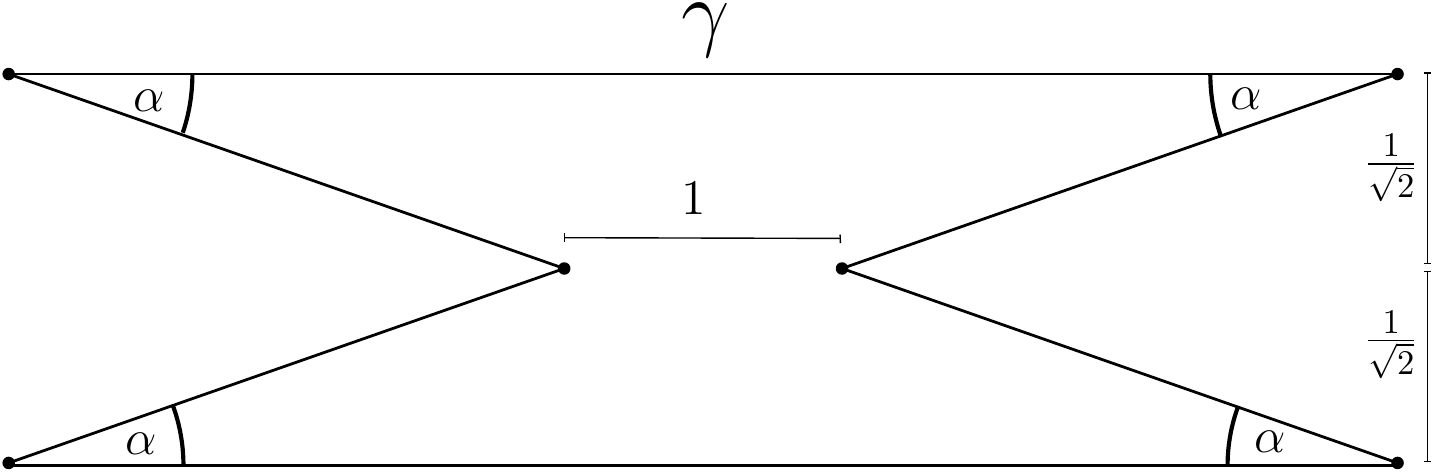}
\caption{The polygon in Theorem \ref{thm:no equilaterals}. The angle $\alpha$ is any angle smaller than $\arcsin(\frac{\sqrt 2}{4})$ radians}
\label{fig:intro}
\end{figure}
Actually we prove that Theorem \ref{thm:no equilaterals} holds in a quantified way.  To make this precise we need  some notation. Given a triangle $\T$ in the plane we introduce the number 
$$\alpha(\T)\in \left[0,\frac{\pi}{3}\right]$$ defined as the minimum of its interior angles, measured in radians. The number $\alpha(\T)$ should be thought of as a measure of the fatness of  $\T.$ In particular 
$\alpha(\T)=\frac \pi3$ if and only if $\T$ is equilateral, while $\alpha(\T)=0$ if and only if $\T$ is degenerate. 
\begin{theorem}\label{thm:no fat}
  There exists a Jordan curve $\gamma$ such that
    \begin{equation}\label{eq:no fat}
         \alpha(\T)\le \arctan(\sqrt 2)
    \end{equation}
     for every triangle $\T $ inscribed in the interior of $\gamma.$
\end{theorem} 
The curve $\gamma$ in Theorem \ref{thm:no fat} can again be taken to be the polygon  in Figure \ref{fig:intro}.
Observe that, since $\arctan(\sqrt 2)<\frac{\pi}{3}$, Theorem \ref{thm:no fat} implies Theorem \ref{thm:no equilaterals}. 
Notably, the angle $\arctan(\sqrt 2)$  ($\sim  54.7^\circ$) is sometimes called \emph{magic angle}  in the literature (see \cite{magic}) and coincides with the angle formed by the diagonal of a cube and an edge adjacent to it.

\begin{remark}[Stronger version of Theorem \ref{thm:no fat}]\label{rmk:stronger}
    Observe that the polygon $\gamma$ in Figure \ref{fig:intro} (and actually any closed polygon) has  many inscribed equilateral triangles $\T$ if we allow one of the sides of $\T$ to be contained in one of the sides of the polygon. Recall however that in our definition of inscribed triangle only the vertices  intersect the boundary. Hence one might wonder if Theorem \ref{thm:no fat} for Jordan curves is still true allowing the sides of inscribed triangles to intersect the curve. Call this type of triangles \emph{weakly inscribed} (see Def.\ \ref{def:w}). It is then easy to slightly modify the example in Figure \ref{fig:intro} to obtain a (non-polygonal) Jordan curve $\gamma$ where all the weakly inscribed triangles are actually inscribed and such that \eqref{eq:no fat} still holds (up to an arbitrary small error)  for  every $\T $  inscribed in the interior of $\gamma$.  This can be done by  slightly bumping inward the sides of the polygon to make them strictly convex.  See Corollary \ref{cor:hourglass} for a precise statement.
\end{remark}
We do not know  if $\arctan(\sqrt 2)$ is the best constant that we can put in \eqref{eq:no fat}, but it is the best constant that we can get with our construction (see Remark \ref{rmk:optimized}). 
Actually it is not at all obvious that the best constant is not zero.  Indeed a priori it could be that for any $\eps>0$ there exists a Jordan curve such that  $\alpha(\T)<\eps$ for every triangle $\T$ inscribed in its interior. Our  second  goal is precisely to prove that this does not happen, by giving a lower bound for the optimal constant in \eqref{eq:no fat}. In other words we show that the interior of every Jordan curve admits an inscribed triangle  which is \emph{not thin} by a definitive universal amount. Actually we will prove that  this is the case for arbitrary open and bounded sets. 
In particular our second main theorem reads as follows.
\begin{theorem}\label{thm:main}
         For any non-empty open and bounded set $\Omega\subset \rr^2$ there exists a  triangle $\T$ inscribed in $\Omega$  satisfying
        \begin{equation}\label{eq:main}
            \alpha(\T)\ge\frac14 \arcsin\left( \frac1{45}\right).
        \end{equation}
        Moreover $\T$ can be taken so that $ A(\T)r(\Omega)^{-2}\ge 10^{-5}$, where $A(\T)$ is the area of $\T$ and $r(\Omega)$ is the maximal radius of a ball contained in $\Omega$ (see Def.\ \ref{def:r}).
\end{theorem}
 Theorem \ref{thm:main} can be seen as an effective improvement of the existence result of Bruckner in Theorem \ref{thm:bruck} above. This result also implies that the constant $\arctan(\sqrt 2)$ in Theorem \ref{thm:no fat} cannot be decreased below $\frac14 \arcsin\left( \frac1{45}\right).$ Our proof of Theorem \ref{thm:main} is  constructive, in the sense that the triangle $\T$ in the statement is built following an explicit algorithm. An additional feature of the construction is that $\T$ can be taken  roughly of size as large  as possible, i.e.\ with area comparable to the largest ball inside $\Omega,$ as specified in the second part of the statement. 
Concerning the value $\frac14 \arcsin\left( \frac1{45}\right)$ ($\sim 0.31^\circ$) in \eqref{eq:main}, it is very likely not-sharp since we did not try to optimize the parameters and the estimates in the argument.

We also present an immediate application of Theorem \ref{thm:main}.
\begin{cor}\label{cor:platform}
    Let $f:\rr^2\to \rr$ be a continuous functions such that the set $\{f<0\}$ is non-empty and bounded. Then there are three points in $\rr^2$ where $f$ vanishes and which are the vertices of a triangle with all angles larger than $0.3$ degrees and with interior  contained in $\{f< 0\}$.
\end{cor}
The above result is related to the famous \emph{table theorem} about placing a  square table with the feet on  an uneven ground, modeled by the graph of continuous function $f:\rr^2 \to \rr$. More precisely this amounts to find four points in $\rr^2$  where $f$ takes the same value and that are the vertices of a square (we refer to  \cite{tab1,tab2,tab3,tab4,tab5} and references therein for more details and versions of the problem).
Corollary \ref{cor:platform} can be thought as a variant where we instead want to  \emph{balance a triangular platform}  on  an uneven ground. Indeed in this case not only the vertices of the platform must be at the same level, but also the rest of the platform must be above ground, which  in the statement corresponds to the interior being contained in $\{f<0\}$.


We conclude this introduction with some further comments and a few open questions. It is relevant to introduce the following constant:
\[
\Theta \coloneqq\inf_{\Omega \subset \rr^2, \text{ open and bounded}}\sup \{ \alpha(\T) \ : \  \T \text{ triangle inscribed in $\Omega$}\}.
\]
Equivalently $\Theta$ is the unique number such that the following hold:
\begin{enumerate}[label=\roman*)]
	\item every non-empty open and bounded set $\Omega\subset \rr^2$ admits an inscribed triangle with all angles $\ge\Theta$ (radians) ,
	\item for every $\beta>\Theta$ there exists an open and bounded set $\Omega\subset \rr^2$ such that every triangle inscribed in $\Omega$ has all angles  $<\beta$ (radians).
\end{enumerate}
Note that  by definition $\Theta\in [0,\frac\pi 3]$,  however our results imply that $\Theta$ assumes a \emph{non-trivial value}, i.e.\ $\Theta\in(0,\frac{\pi}{3})$.   Indeed combining Theorem \ref{thm:no fat} and Theorem \ref{thm:main} we obtain the following explicit bounds for $\Theta.$
\begin{cor}
\begin{equation}\label{eq:theta bounds}
    \frac14 \arcsin\left( \frac1{45}\right)\le \Theta\le  \arctan(\sqrt2).
\end{equation}
\end{cor}
This motivates asking the following questions, which we believe interesting for future investigations.
\begin{open}
  What is the value of the constant  $\Theta$? Alternatively, can we improve the bounds in \eqref{eq:theta bounds}?
\end{open}
\begin{open}
 Does it exist an open and bounded set $\Omega \subset \rr^2$ that realizes $\Theta,$ i.e.\ such that $\alpha(\T)\le \Theta$ for every triangle $\T$ inscribed in $\Omega$?
\end{open}

The paper is structured as follows. In Section \ref{sec:pre} we fix some notations and gather some preliminaries. In Section \ref{sec:hourglass} we 
prove Theorems \ref{thm:no equilaterals} and \ref{thm:no fat}. The proof of Theorem \ref{thm:main} will be given in Section \ref{sec:main}, except for a technical lemma about functions on the real line which is postponed to Section \ref{sec:lemma}.

\medskip

\textbf{Acknowledgements.} 
We are grateful  to  Petri Juutinen for proposing Question \ref{q:equilateral}, which initially motivated us to study the properties of inscribed triangles.
The author also thanks Katrin F\"assler and Tuomas Orponen for fruitful discussions on the topic.
The  author is supported by the Academy of Finland project Incidences on Fractals, Grant No.321896.

\section{Preliminaries and notations}\label{sec:pre}
We gather in this section some of the notations, conventions and elementary results that we will use in the sequel. 

The symbols $\cl(C)$, $\Int(C)$ and $\partial C$ will indicate respectively the topological closure, interior and boundary of  a set $C\subset \rr^2$. Also with $C^c$ we will denote its complement $\rr^2\setminus C.$
Given two points $A,B\in \rr^2$ we will denote by  $AB$ the closed segment with endpoints $A$ and $B$ and we with $\overline{AB}$ its length. 
Given three points $A,B,C\in \rr^2$, the triangle  $\T$ of vertices $A,B,C$ will denote the closed convex hull of $\{A,B,C\}$. By $A(\T)$ we will indicate the area of $\T.
$ We will also denote by $\hat A,\hat B,\hat C$ the interior angles  of $\T$ (measured in radians) respectively at the vertices $A,B,C.$ 

Often we will implicitly  fix a coordinates-system of $\rr^2$, i.e.\ $\rr^2=\{(x,y) \ : \ x,y \in \rr\}$. Then, with a slight abuse of notation, we will often write $\{x=a\}$, $a\in \rr$, to denote the vertical lines $\{(x,y) \in \rr^2 \ : x=a \ \}$ and similarly for the horizontal lines $\{y=a\}$, the strips $\{a<x<b\}$, $\{a\le x<b\}$, $\{a<y<b\}$, the half planes $\{y<a\}$, $\{x\ge a\}$ and so on.
Given an affine function $l:\rr \to \rr$ we will write $y=l(x)$, or just $l$, to denote the non-vertical line  $\{(x,y) \in \rr^2\ : \ y=l(x)\}$ given in Cartesian coordinates. 
We will also denote by $l_\le$ and $l_<$ respectively the open and closed lower half planes determined by $l$, that is $l_\le= \{(x,y)\in \rr^2 \ : \ y\le l(x)\}$ and $l_< =\{(x,y)\in \rr^2 \ : \ y< l(x)\}.$

By $\la x,y\ra$ we will denote the scalar product between points $x,y \in \rr^2$.

For every $r>0$ and $x\in \rr^2$ we defined the open ball of radius $r$ centred at $x $ by  $B_r(x)\coloneqq\{y \in \rr^2 \ : \ |x-y|<r\}$. Additionally for any set $C\subset \rr^2$ we define $\sfd(x,C)\coloneqq \sup \{|x-y|\ :\ y \in C\}.$

We will say that a closed rectangle  $\mathcal R\subset \rr^2$ is \emph{of sides-length} $A\times B$, for some $A,B \in (0,\infty)$, if  $\mathcal R=[a,b]\times[c,d]$ with $|a-b|=A$ and $|c-d|=B$, with respect to some choice of coordinates-system.

Next we introduce  the maximal radius of a ball contained in an open set.
\begin{definition}[Maximal inscribed radius]\label{def:r}
    For every $\Omega \subset \rr^2$ open and bounded define the number
    \begin{equation}\label{eq:inradius}
         r(\Omega)\coloneqq \sup \{r>0 \ : \ \text{there exists $B_r(x)\subset \Omega$}\} \in (0,\infty). 
    \end{equation}
\end{definition}
The following elementary observation will be useful.
\begin{lemma}[Existence of maximal inscribed ball]\label{lem:inscribed ball}
    Let $\Omega \subset \rr^2$ be a non-empty bounded open set. Then the supremum in the definition of $r(\Omega)$  (see \eqref{eq:inradius}) is in fact a maximum. Moreover for every ball $B_r(x)\subset \Omega$ with $r=r(\Omega)$ (i.e.\ which attains this maximum) and every line $l$ containing $x$, the set $\partial B_r(x)\cap \partial \Omega$ is not contained in any of the open half planes determined by $l.$
\end{lemma}
\begin{proof}
    Let $f:\cl(\Omega) \to [0,\infty)$ be given by $f(x)\coloneqq \sfd(x,\partial \Omega)$. Then by definition $\sup_{\cl(\Omega)}f=r(\Omega)>0$. Since $f$ is continuous on a compact set, its maximum is attained at some point $x$, which must lie in $\Omega$ (otherwise $f(x)=0$). This proves the first part. Consider now a ball $B_r(x)\subset \Omega$ with $r=r(\Omega)$ and define the compact set $K\coloneqq \partial B_r(x)\cap \partial \Omega.$ Let $l$ be any line containing $x$ and denote by $H$, $H'$ the two open half planes determined by $l$. Suppose by contradiction that $K$ is contained in $H.$ Then also a suitable open neighbourhood $U$ of $K$ is contained in $H$. The set $\partial \Omega \cap U^c$ is closed and  by definition of $K$ it is contained in $ \rr^2\setminus \cl(B_r(x))$, therefore 
    \begin{equation}\label{eq:dist > r}
        \sfd(x,\partial \Omega \cap U^c)>r,
    \end{equation}
      Let now $h$ be the half-line starting from $x$, orthogonal to $l$ and contained in   the closure of $H'$, i.e.\ the half space that  does not contain $K$. Then for every  point $x'\in h$ with $x\neq x'$ it holds 
      \begin{equation}\label{eq:dist > r v2}
          \sfd(x',\partial \Omega \cap U)> \sfd(x,\partial \Omega \cap U)=r(\Omega)
      \end{equation}
       Indeed, since $x'\in H'$, it holds $\la x-y,x'-x\ra\ge 0$ for all $y\in H$ which implies
      \[
      |x'-y|^2=|x-y|^2+|x-x'|^2+2\la x-y, x-x'\ra\ge |x-y|^2+|x-x'|^2, \quad \forall y \in H.
      \]
      Hence \eqref{eq:dist > r v2} holds because $U\subset H.$
      On the other hand by  \eqref{eq:dist > r} and continuity, if $x'$ is sufficiently close to $x$, we have $\sfd(x',\partial \Omega \cap U^c)>r$. Combining this with \eqref{eq:dist > r v2} shows that $\sfd(x',\partial \Omega)>r(\Omega)$ for $x'$ sufficiently close to $x$, which is a contradiction.
\end{proof}

\section{A polygon without inscribed fat  triangles}\label{sec:hourglass}
Here we will construct Jordan curves such that all the triangles inscribed in their interior are not too fat, or in other words   all their angles  are less than some constant smaller than $\frac{\pi}{3}$. In particular they do not admit any inscribed equilateral triangle, which in our notation is the fattest triangle possible. 

In the argument we will repeatedly use the following elementary inequality useful to estimate $\alpha(\T)$ and that we isolate here for convenience.
\begin{lemma}\label{lem:equivalent fatness}
    Let $\T$ be a triangle of vertices $A,B,C\in \rr^2$ and let $CH$ be the height of base $AB$. Then
\begin{equation}\label{eq:ratio bh}
    \alpha(\T)\le \arctan\left(\frac{2\overline{CH}}{\overline{AB}} \right).
\end{equation}
\end{lemma}
\begin{proof}
If $H$ belongs to the (closed) segment $AB$, then at least one between $\overline{AH}$ and $\overline{BH}$ is greater or equal than $\overline{AB}/2$, since $\overline{AH}+\overline{BH}\ge \overline{AB}$. 
We can assume that $\overline{BH}\ge \overline{AB}/2$. Since $\hat B\le \pi/2$ this gives $\hat B=\arctan(\overline{CH}/\overline{BH})\le \arctan(\frac{2\overline{CH}}{\overline{AB}} ),$ which shows \eqref{eq:ratio bh}. If instead $ H$ does not belong to $AB$ then we can assume that the points $H,A,B$ lies in the extended side of $AB$ in this precise order. In particular  $\overline{BH}\ge \overline{AB}$ and $\hat B\le \pi/2$. This gives $\hat B=\arctan(\overline{CH}/\overline{BH})\le \arctan(\frac{\overline{CH}}{\overline{AB}} ),$ which implies again \eqref{eq:ratio bh}.
\end{proof}

We are now ready to prove Theorem \ref{thm:no fat} that we restate here for the convenience of the reader. 
\begin{theorem}\label{thm:hourglass}
     There exists a closed polygonal curve $\mathcal H$ such that
    \begin{equation}\label{eq:upper bound fat}
         \alpha(\T)\le \arctan(\sqrt 2)
    \end{equation}
     for every triangle $\T $ inscribed in the interior of $\mathcal H.$
\end{theorem}
\begin{proof}
The hourglass-shaped polygon $\mathcal H$ is constructed as follows. Fix any angle $\alpha$ small enough, in particular
$$\alpha\le \min(\arctan(\sqrt 2), \arcsin(\sqrt 2/4), \arctan (1/\sqrt 2) )=\arcsin(\sqrt 2/4)$$
will suffice.
Then  $\mathcal H$ is defined as the  six-sided polygon of vertices $A,B,C,D,E,F$:
\[
\begin{split}
    &A=\left(-\frac12,0\right),\quad  B=\left(\frac12,0\right),\\
    C=\left(-L,-\frac1{\sqrt 2}\right),\quad &D=\left(L,-\frac1{\sqrt 2}\right),\quad  E=\left(-L,\frac1{\sqrt 2}\right),\quad F=\left(L,\frac1{\sqrt 2}\right).
\end{split}
\]
  where $L>0$ satisfies
\[
\hat C=\hat D=\hat E=\hat F=\arctan\left(\frac{\sqrt 2}{L-\frac12}\right)=\alpha
\]
(see Figure \ref{fig:hourglass}). The  constant $\frac{1}{\sqrt 2}$ above is for optimization reasons, in particular with our construction this value gives the best upper bound in \eqref{eq:upper bound fat} (see also Remark \ref{rmk:optimized}).

\begin{figure}[!ht]
\centering
\includegraphics[width=12cm, height=4cm]{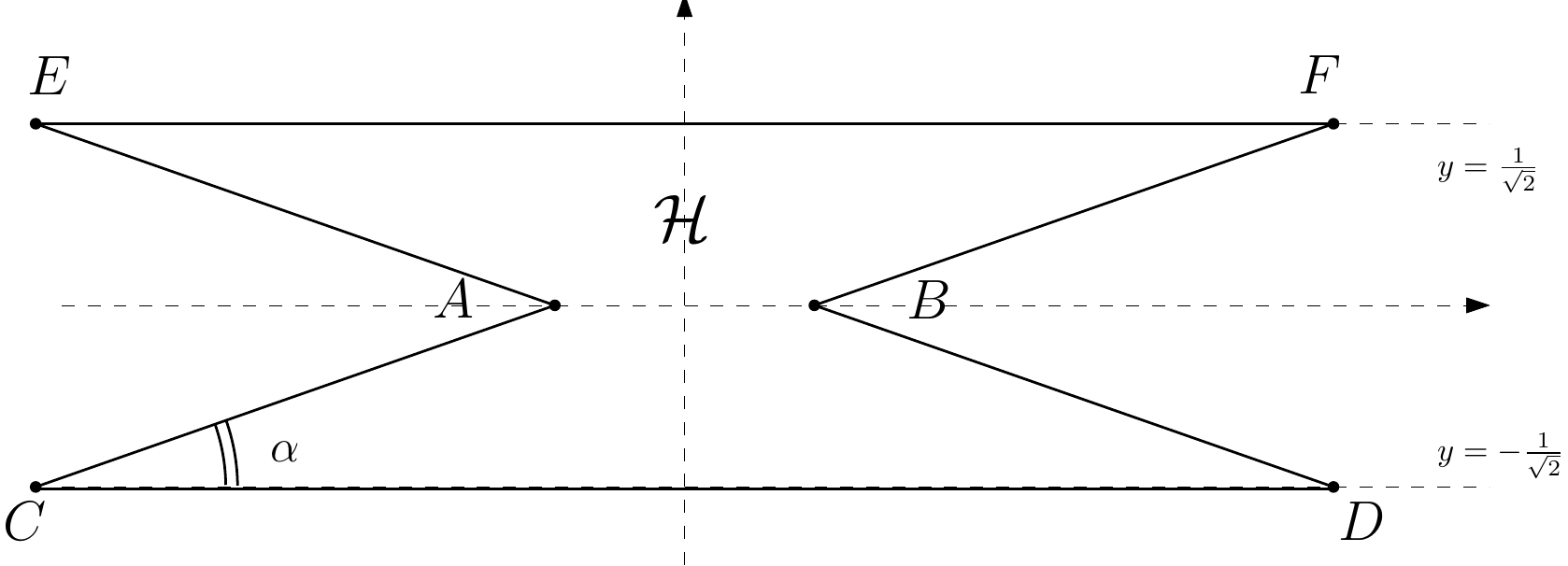}
\caption{The construction of $\mathcal H$}
\label{fig:hourglass}
\end{figure}

We need to show that every triangle $\T$  inscribed in the interior of $\mathcal H$ satisfies \eqref{eq:no fat}.
Recall that if $\T$ is inscribed in the interior of $\mathcal H$, then it intersects the boundary of $\mathcal H$ only at the vertices and \emph{not} in the interior of its sides. In particular:
\begin{equation}\label{eq:one vertex}
    \parbox{13cm}{\centering each side of $\mathcal H$ contains at most one vertex of $\T$. }
\end{equation}
We stress that here by side we mean \emph{closed side}, i.e.\ containing also the vertices, in particular the sides of $\mathcal H$ are not pairwise disjoint.

Let $\T$ be a triangle inscribed in  the interior of $\mathcal{ H}$  and denote by $P_i=(x_i,y_i)$, $i=1,2,3,$ the vertices of $\T.$ For convenience we will call the (closed) sides $EF$ and $CD$ the \emph{horizontal sides} of $\mathcal H$ and  \emph{non-horizontal sides} the remaining ones.

We prove \eqref{eq:upper bound fat} by distinguishing several cases (see Figure \ref{fig:cases}).
\begin{figure}[!ht]
\centering
\includegraphics[width=14cm, height=15cm]{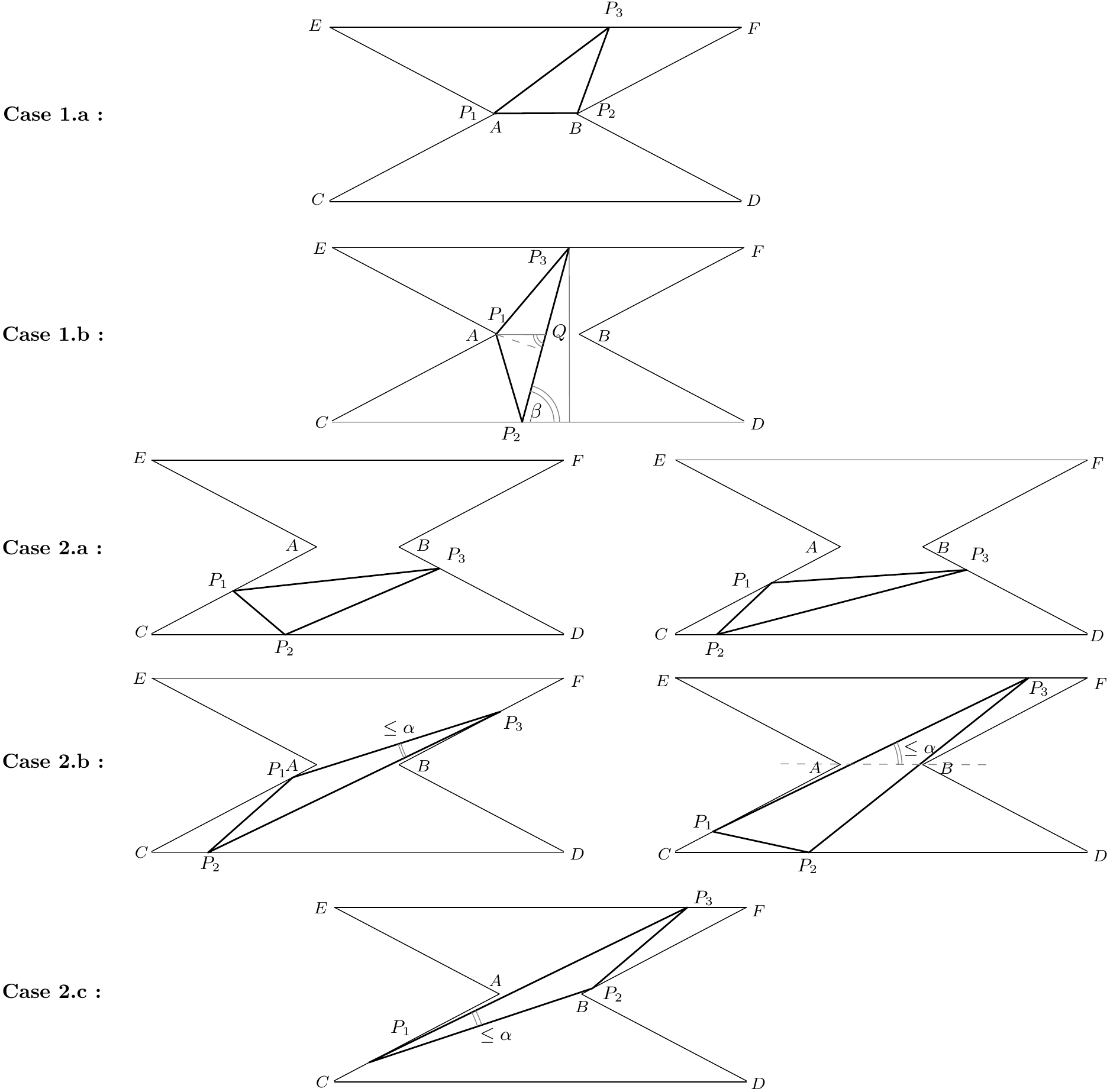}
\caption{The different cases in the proof of Theorem \ref{thm:hourglass}}
\label{fig:cases}
\end{figure}

\noindent{\textbf{Case 1}:} None of $P_1,P_3,P_3$ lies in $\{-1/\sqrt{2}<y<0\}\cup \{0<y<1/\sqrt{2}\}$, i.e.\ none of them lies in the interior of one of the non-horizontal sides. 

There are two possibilities, up to renaming the points: $P_1=A,P_2=B$ and $P_3$ lies in one of the horizontal sides; $P_1=A$ (or $P_1=B$) and $P_2,P_3$ both lie in (different) horizontal sides. 

\medskip

\noindent{\textbf{Case 1.a}:}  $P_1=A,P_2=B$ and $P_3$ lies in one of the horizontal sides. We have $\overline{P_1P_2}=1$ and that the height of base $P_1P_2$ measures exactly $1/\sqrt 2.$ Therefore by Lemma \ref{lem:equivalent fatness}
\[
\alpha(\T)\le \arctan(2\sqrt{2}^{-1})=\arctan(\sqrt{2}).
\]

\medskip

\noindent{\textbf{Case 1.b}:} $P_1=A$  and $P_2,P_3$ both lie in (different) horizontal sides. 
Call $Q$ the intersection between $P_2P_3$ and the horizontal axis $\{y=0\}$. The key point is that $\overline{P_1Q}\le 1$, otherwise $Q$ must lie outside $\mathcal H$ and the triangle $\T$ would not be inscribed in the interior of $\mathcal H$. Denote by $\beta$ the (smaller)  angle that the segment $P_2P_3$ forms with the horizontal axis. Then  $\overline{P_2P_3}=\sqrt 2 \sin(\beta)^{-1}$ and the height of base $P_2P_3$ measures $h=\overline{P_1Q}\sin(\beta)\le \sin(\beta)$. Therefore by Lemma \ref{lem:equivalent fatness}
\[
\alpha(\T)\le \arctan(2h/\overline{P_2P_3})\le \arctan\left (\frac{2\sin(\beta)^2}{\sqrt 2}\right )\le  \arctan(\sqrt 2).
\]

\medskip

\noindent{\textbf{Case 2}:} At least one of $P_1,P_3,P_3$ lie in $\{-1/\sqrt{2}<y<0\}\cup \{0<y<1/\sqrt{2}\}$, i.e.\ up to symmetries and relabelling:
\begin{equation}
    \text{$P_1\in AC\setminus\{A,C\}$.}
\end{equation}

\medskip

\noindent{\textbf{Case 2.a}:} $P_2,P_3 \in \{y\le 0\}$.

Again up to symmetries and relabelling, since each side of $\mathcal H$ contains at most one of $P_1,P_2,P_3$, we can assume that $P_1\in AC$, $P_2\in CD$, $P_3 \in DB$. We have that $\overline{P_1P_3}\ge \overline{AB}=1$. If $x_1\le x_2\le x_3$ (recall that $P_i=(x_i,y_i)$), then the height of base $P_1P_3$  has length  less  than $\frac{1}{\sqrt 2}.$ If instead $x_2\le x_1\le x_3$, then $\overline{P_2P_3}\ge 1$ and the height of base $P_2P_3$    has length  less  than $\frac{1}{\sqrt 2}.$ Symmetrically, if instead $ x_1\le x_3\le x_2$, then $\overline{P_2P_1}\ge 1$ and the height of base $P_2P_1$    has length  less  than $\frac{1}{\sqrt 2}.$ In all cases $\T$ has a height $h$ over a base $b$ such that
$h/b\le\frac{1}{\sqrt 2}.$ Therefore by Lemma \ref{lem:equivalent fatness}
\[
\alpha(\T)\le \arctan(\sqrt 2).
\]

\medskip

\noindent{\textbf{Case 2.b}:}  $P_2\in \{y\le 0\}$, $P_3\in \{y> 0\}$ (or equivalently $P_3\in \{y\le 0\}$, $P_2\in \{y>0\}$).

$P_3$ must lie either in $BF$ or $EF$, otherwise if $P_3$ was on $EA$ then $P_1P_3$ would intersect the exterior of $\mathcal H$ and $\T$ would not be inscribed in $\mathcal H$. If $P_3 \in BF$, then $\hat{P_3}\le \alpha\le  \arctan(\sqrt{2})$. So we assume $P_3 \in EF$. Note that, since $P_1\neq A$, the segment $P_1P_3$  forms an angle not bigger than $\alpha $ with the horizontal axis (otherwise  $P_1P_3$ would intersect the exterior of $\mathcal H$). Therefore  $\overline{P_1P_3}\ge |y_1-y_3|\sin(\alpha)^{-1}\ge \frac1{\sqrt 2\sin(\alpha)}.$ Similarly, since the segment $P_1P_3$ intersects the horizontal axis $\{y=0\}$ at some $x\ge -1/2$ we have $ x_3\ge  \frac1{\sqrt 2\tan(\alpha)}-\frac12$. Hence, since  $\tan(\alpha)<(1/\sqrt 2)$,  it must be that $x_3>1/2$. This implies that $x_2\le x_3$, otherwise $P_2P_3$ would intersect the horizontal  axis $\{y=0\}$ at some $x>1/2$ (and thus outside $\mathcal H$) and $\T$ would not be inscribed in $\mathcal H$. Therefore the vertical line through $P_2$ intersects the line containing $P_1P_3$ in $\{y\le \frac{1}{\sqrt 2}\}$ and so the height of base $P_1P_3$ is less than $2\cdot \frac1{\sqrt 2}=\sqrt{2}.$ Applying Lemma \ref{lem:equivalent fatness} and using the above lower bound for $\overline{P_1P_3}$ we get
\[
\alpha(\T)\le \arctan\left (4\sin(\alpha)\right)\le \arctan(\sqrt 2),
\]
because of the assumption $\alpha\le \arcsin(\sqrt 2/4)$.

\medskip

\noindent{\textbf{Case 2.c}:}  $P_2,P_3\in \{y>0\}$.

We must have that $\widehat P_1\le \alpha \le \arctan(\sqrt 2)$ (as in the first part of \textbf{Case 2.b}).
\end{proof}

\begin{remark}\label{rmk:optimized}
    The choice of the parameters in the construction of $\mathcal H$  in the proof of Theorem \ref{thm:hourglass} is optimal  in the following sense that. Fixed the points $A,B$ at unit distant and choosing the points $E,F$ and $CD$ respectively on the horizontal lines $\{y=a\}$ and $\{y=-a\}$ (as in Figure \ref{fig:hourglass}), the best upper bound in \eqref{eq:upper bound fat} is obtained with $a=\frac{1}{\sqrt 2}.$
\end{remark}

As anticipated in the Introduction (see Remark \ref{rmk:stronger}) if we consider more general Jordan curves than polygonal ones, we can prove a  stronger version of Theorem \ref{thm:hourglass} where the sides of the inscribed triangles, and not only the vertices, are allowed to touch the boundary. More precisely we introduce:
\begin{definition}[Weakly inscribed triangles]\label{def:w}
    A triangle $\T$ is \emph{weakly inscribed} in an open set $\Omega$ if the  interior of $\T$ is contained in $\Omega$ and if all the vertices of $\T$ lie in $\partial \Omega.$
\end{definition}
For example any non-degenerate triangle $\T$ is weakly inscribed, but \emph{not inscribed}, in its own interior.

As a consequence of Theorem \ref{thm:hourglass} we can obtain the following.
\begin{cor}\label{cor:hourglass}
     For every $\eps>0$, there exists a Jordan curve $\mathcal H_{\eps}$ such that 
   \begin{equation}\label{eq:upper bound fat strong}
        \alpha(\T)\le \arctan(\sqrt 2)+\eps,
   \end{equation}
   for every triangle $\T $  weakly  inscribed in the interior of $\mathcal H_{\eps}$.
\end{cor}
\begin{proof}
The idea of the construction of $\mathcal H_{\eps}$ is to  modify the polygon $\mathcal H$ given in Figure \ref{fig:hourglass} by slightly ``bumping'' the sides inwards, i.e.\ by replacing them with inward-facing circular arcs that are close enough to the original sides.
Concretely, for every $\eps>0$ we construct the curve $\mathcal H_{\eps}$ as follows: replace each side of $\mathcal H$ by a circular arc of height $\eps$ having as chord  that side of $\mathcal H$.  In particular $\mathcal H_{\eps}$ is the union of  \emph{ six inward-facing circular arcs} (see Figure \ref{fig:hourglass2}). From now on we call these circular arcs (together with their endpoints) the \emph{sides} of $\mathcal H_{\eps}.$ 
\begin{figure}[!ht]
\centering
\includegraphics[width=8cm, height=3cm]{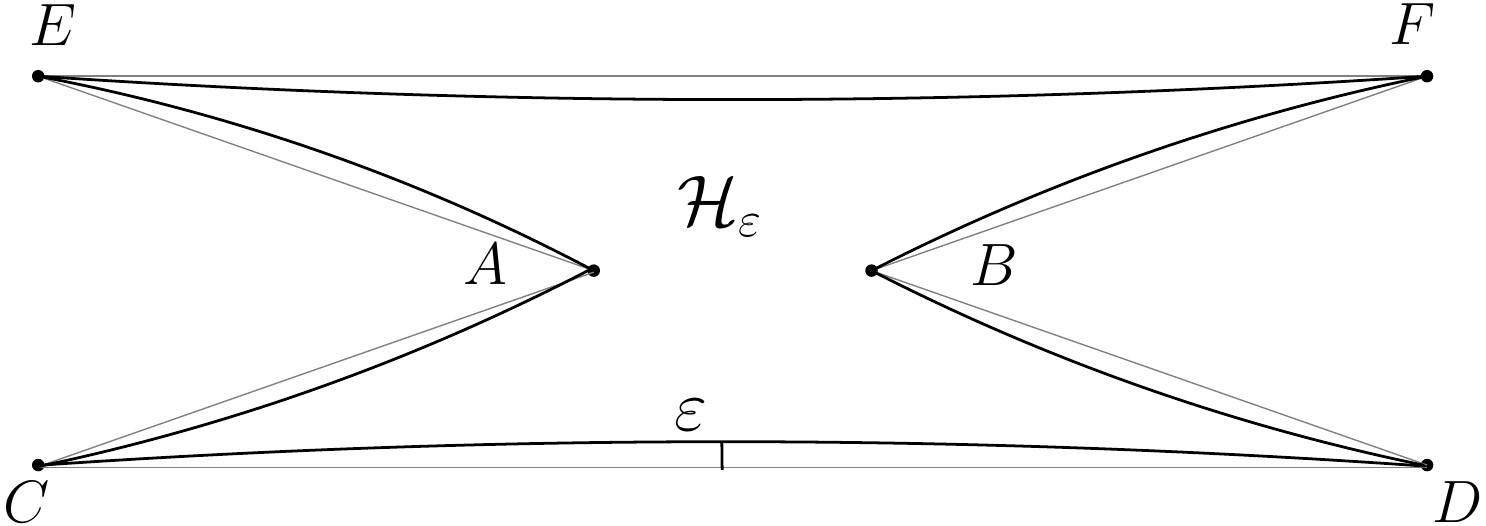}
\caption{The construction of $\mathcal H_{\eps}$}
\label{fig:hourglass2}
\end{figure}

We prove the result arguing  by contradiction. Suppose that the statement is false for some $\eps>0.$ Fix a sequence $\eps_n\to 0$ with $\eps_n>0.$ Then for every $n$ there must be  triangle $\T_n$ inscribed in $\mathcal H_{\eps_n}$ such that $\alpha(\T_n)>\sqrt{2}+\eps.$ Since   the sides of $\mathcal H_{\eps_n}$ are strictly convex, each of them contains at most one vertex of $\T_n$. In particular there exists some $\delta>0$, independent of $n$, such that at least one of the sides of $\T_n$ is bigger than $\delta>0.$ Hence by compactness, up to passing to a subsequence, the vertices of $\T_n$ converge to three vertices of a  triangle $\T$, which is not a single point, and so must satisfy $\alpha (\T)\ge \arctan(\sqrt 2)+\eps$.  Since the vertices of $\T_n$ belong to the sides of $\mathcal H_{\eps_n}$, the vertices of $\T$ belong to sides of $\mathcal H.$ Moreover since the interior of $\mathcal H_{\eps_n}$ is contained in the interior of $\mathcal H$, the interior of $\T$ is contained in the interior of $\mathcal H$.  Finally, recalling that the each of the side of $\mathcal H_{\eps_n}$  contains at most one vertex of $\T_n$, we deduce that no couples of vertices of $\T$ can lie in the interior of the same side of $\mathcal H$ (i.e.\ the side minus the endpoints), otherwise for $n$ big enough this would imply that $\T_n$ has two vertices on the same side of $\mathcal H_{\eps_n}$. Therefore there are two possibilities: either each side of $\mathcal H$ contains at most one vertex of $\T$; or two of the vertices of $\T$ (that we call $P,Q$) belong to the same side of $\mathcal H$, but at least one of them (say $P$) coincides with one of its  endpoints. In first case  $\T$ is inscribed in $\mathcal H,$ hence  Theorem \ref{thm:hourglass} implies that $\alpha(\T)\le \arctan(\sqrt 2)$, which gives a contradiction. In the second case $P$ is an endpoint, i.e.\ $P \in \{A,B,C,D,E,F\}$, and $Q$ belong to a side adjacent to $P$. If $P\in \{C,D,E,F\}$ then clearly $\alpha (\T)\le \alpha\le \arctan(\sqrt 2),$ which is a contradiction. Hence (up to symmetry) we can suppose that $P=A$ and  that $Q$ belong to $ AC$. If $Q\in \{A,C\}$  then as before $\alpha (\T)\le \alpha\le \arctan(\sqrt 2),$ hence we can assume that $Q$ is in the interior of $AC.$ By construction there exist $P_n, Q_n$, vertices of $\T_n$, such that $P_n\to P$ and $Q_n \to Q.$ However, since $P_n,Q_n$ do not lie both in the same side of  $\mathcal H_{\eps_n}$ we must have that (for $n$ big enough) $Q_n$ lies in the interior of the arc $AC$ (in $\mathcal H_{\eps_n}$) and $P_n$ lies in the interior of the arc $AE.$ However, by the construction of $\mathcal H_{\eps_n}$ (see Figure \ref{fig:hourglass2}), this clearly implies that the segment $P_nQ_n$ intersects the exterior of $\mathcal H_{\eps_n}$, which contradicts the fact that $\T_n$ is inscribed in $\mathcal H_{\eps_n}$,
\end{proof}

We end this part with some open questions concerning possible improvements on Theorem \ref{thm:hourglass} and Corollary \ref{cor:hourglass}.
\begin{open}
    Is it possible to take $\mathcal H$ in Theorem \ref{thm:hourglass} to be a polygon with at most five sides?  (In our case it has  six).
\end{open}
\begin{open}
    Is it possible to take $\mathcal H_{\eps}$ in Corollary \ref{cor:hourglass} to be smooth curve, e.g.\ $C^\infty$ or $C^2$?  
\end{open}

\section{Existence of fat inscribed triangles in arbitrary domains}\label{sec:main}
This section is devoted to the proof of Theorem \ref{thm:main}. First  we describe a rough outline of the argument. 
Recall that the ultimate goal is to find a sufficiently fat triangle inscribed in any bounded open set $\Omega$. The main obstacles to do so are open sets which are very long and thin. The first easier step of the proof will be then to find fat triangles inscribed in open sets that are not too thin. More precisely we show that either there is an inscribed fat triangle or the open set $\Omega$ contains a \emph{very long rectangle} (see Proposition \ref{prop:fat or rectangle}). Long here means that the ratio between the short and the long side is very large and the shorter side is comparable to $r(\Omega)$ (the maximal inscribed radius). The second and more involved step will be to prove that, given the presence of a long rectangle, either we can find a fat triangle or $\Omega$ contains {\emph{another fatter long rectangle}, i.e.\ having the short side  slightly longer  than the previous rectangle (see Proposition \ref{prop:cases}). Since these rectangles can not get fatter and fatter, iterating this statement we will eventually yield the presence of an inscribed fat triangle.

We start by giving the statements of the two main propositions that we just described, from which the main theorem will easily follow. Recall that $r(\Omega)$ denotes the maximal inscribed radius in $\Omega$ (see Def.\ \ref{def:r}).
\begin{prop}[Fat triangle or long rectangle]\label{prop:fat or rectangle}
Let $C_1\ge 1$ be any constant. Let $\Omega$ be any bounded open set with $r(\Omega)=1$. Then at least one of the following holds:
 \begin{enumerate}
     \item[i)] there exists a triangle $\T$ inscribed in $\Omega$ with $\alpha(\T)\ge  \frac14 \arcsin\big(\frac{1}{3C_1}\big)$ and  of area $A(\T)\ge \sin^2\left(\frac14\arcsin( \frac{1}{3C_1})\right),$
     \item[ii)]  there exists a (closed) rectangle $\mathcal R\subset \Omega$ of sides-length $2a\times 2A$ with $A\ge C_1$ and $a=2/3.$
 \end{enumerate}
\end{prop}

\begin{prop}[Fat triangle or fatter rectangle]\label{prop:cases}
		Let $\Omega\subset \rr^2$ be open and bounded with $r(\Omega)=1$. Suppose there exists a closed rectangle ${\mathcal{R}}\subset \Omega$ of sides $2A\times 2a$, with $A\ge 13$, $a\in [2/3,1]$.  Then at least one of the following holds:
 \begin{enumerate}
     \item[i)] there exists a triangle $\T$ inscribed in $\Omega$, with $\alpha(\T)\ge \arctan\left(\frac{1}{47}\right)$ and of area
     $A(\T)\ge  \frac{9}{4420},$
     \item[ii)]  there exists a (closed) rectangle  $\tilde {\mathcal{R}}\subset \Omega$ of sides-length $2\tilde A\times 2\tilde a$, with $\tilde a \le 1$ and 
		\begin{equation}\label{eq:Acontrol}
			\begin{split}
				& \tilde A- A\ge 2{\sqrt 2}(a-\tilde a),\quad \frac{1}{3\diam(\Omega)} \le 2(\tilde a-a).
			\end{split}
		\end{equation}
 \end{enumerate}
\end{prop}
We stress that the key point of Proposition \ref{prop:cases} are inequalities \eqref{eq:Acontrol}. The second  one says that in the new rectangle $\tilde {\mathcal R}$ the short side has increased by a definitive amount (i.e.\ $\frac{1}{3\diam(\Omega)}$), while the first one says that the longer side $\tilde A$ might me smaller, but (roughly) can not decrease more than how much the short side has increased. Note in particular that the right hand side of the first in \eqref{eq:Acontrol} is negative.

We now show how to prove Theorem \ref{thm:main} given the validity of the above results. Note that, even if the argument is by contradiction, it amounts to prove that a certain algorithm ends in a finite number of steps. Therefore the proof  still provides an explicit way of constructing a fat inscribed triangle in any given open set.
\begin{proof}[Proof of Theorem \ref{thm:main}]
	 Suppose by contradiction that
 \begin{equation}\label{eq:absurdum}
     \alpha(\T)<\frac14 \arcsin\left( \frac1{45}\right), \quad \text{ }
 \end{equation}
 for every  triangle $\T$ inscribed in $\Omega$ and having area $A(\T)\ge 10^{-5}r(\Omega)^2.$
 Up to rescaling we can assume that $r(\Omega)=1$.  We start by applying  Proposition \ref{prop:fat or rectangle}, with $C_1=15$. Then by \eqref{eq:absurdum} and since $\sin^2\left(\frac14\arcsin( \frac{1}{45})\right)\ge 10^{-5}$, we deduce that there must be a closed rectangle $\mathcal R_0\subset \cl(\Omega)$ of sides-length $2a_0\times 2A_0$ with $A_0\ge 15$ and $a_0=2/3.$
Next we claim that there exists an infinite sequence of closed rectangles $\mathcal R_i\subset \Omega$, $i \in \nn$, of sides-length $2a_i\times 2A_i$ satisfying for all $i \in \nn$:
 \begin{equation}\label{eq:induction prop}
 \begin{split}
 &\frac23 \le a_i\le 1,\\
     & A_{i}-A_{i-1}\ge 2{\sqrt 2}(a_{i-1}-a_{i}),\\
		&\frac{1}{3\diam(\Omega)} \le 2(a_{i}-a_{i-1}).
 \end{split}
 \end{equation}
We prove this by induction. To build $\mathcal R_1$ observe that, since $A_0\ge 13$ and $a_0=\frac23$,  we can  apply Proposition \ref{prop:cases} with $\mathcal R=\mathcal R_0$.  Then, thanks to the contradiction assumption \eqref{eq:absurdum} and since
\begin{equation}\label{eq:ok}
\begin{split}
      &\frac14 \arcsin\left( \frac1{45}\right)=0.005 \ldots  \le \arctan\left(\frac1{47}\right)=0.02\ldots,\\
      & \frac{9}{4420}\ge 10^{-5},
\end{split}
\end{equation}
we deduce that $ii)$ of Proposition \ref{prop:cases} must hold, i.e.\ there exists another closed rectangle $\mathcal R_1\subset \Omega$ of sides-length $2a_1\times 2A_1$ satisfying $a_1\le 1$ and 
\begin{equation}\label{eq:step 1 induction}
    A_1-A_0\ge 2{\sqrt 2}( a_0-a_1),\quad		\frac{1}{3\diam(\Omega)} \le a_1-a_0.
\end{equation}
 In particular $a_1\ge a_0\ge 2/3.$ This proves  \eqref{eq:induction prop} for $i=1.$  Suppose now that we have constructed closed rectangles $\mathcal R_i\subset \Omega$ with $i=1,\dots, k$ for some $k \in \nn$, of sides-length $2a_i\times 2A_i$ and satisfying \eqref{eq:induction prop} for every $i \in \{1,\dots k\}.$  Then we have
\[
A_k=A_0+\sum_{i=1}^{k} A_{i}-A_{i-1}\ge A_0+\sum_{i=1}^{k} \frac{3(a_{i-1}-a_{i})}{\sqrt 2}=A_0+ 2{\sqrt 2}(a_0-a_k)\ge 15-\frac{2{\sqrt 2}}{3}\ge 13,
\]
where in the second to last inequality we used that $a_0=\frac23$ and $a_k\le 1$.
Therefore the assumption of Proposition \ref{prop:cases} are again satisfied taking $\mathcal R=\mathcal R_k$. Hence,  by the contradiction assumption \eqref{eq:absurdum} and \eqref{eq:ok}, we deduce that item $ii)$ of Proposition \ref{prop:cases} must hold, i.e.\ there exists a closed rectangle $\mathcal R_{k+1}\subset \Omega$ of sides-length $2a_{k+1}\times 2A_{k+1}$ satisfying \eqref{eq:induction prop} for $i={k+1}.$ This proves the claim, i.e.\ that there exists a sequence of closed rectangles $\mathcal R_i\subset  \Omega$, $i \in \nn$, of sides-length $2a_i\times 2A_i$ satisfying \eqref{eq:induction prop} for all $i \in \nn$. However using the first and last in \eqref{eq:induction prop} we can write
\[
1\ge a_i= a_0+\sum_{j=1}^{i} a_{j}-a_{j-1}\ge a_0 + \frac{i}{6\diam(\Omega)}, \quad \forall i \in \nn,
\]
which is a contradiction.
\end{proof}

It remains to prove Propositions \ref{prop:fat or rectangle} and \ref{prop:cases}. 
We start with the first one, which is considerably easier. 
\begin{proof}[Proof of Proposition \ref{prop:fat or rectangle}]
Fix $C_1\ge 1$ arbitrary and set $\delta_1=\frac12\arcsin( \frac{1}{3C_1})<\frac \pi 2$.

By Lemma \ref{lem:inscribed ball} there exists $B_1(x)\subset \Omega$. Up to a translation we can assume that $x=0$ is the origin.
Lemma \ref{lem:inscribed ball}  also says that $C\coloneqq \partial B_1(0)\cap \partial \Omega$ can not be contained in any open half plane determined by a line through the origin. In particular $C$ contains at least two points $P,Q$. Up to a rotation we can assume that $P=(0,-1)$. Moreover, since $C$ is \emph{not} contained in the open lower half plane $\{y<0\}$, we can assume that $Q$ is in the (closed) upper half plane $\{y\ge 0\}$. We also denote by $\bar P=(0,1)$ the antipodal point of $P$ in $\partial B_1(0).$ We now distinguish two cases depending on how much $Q$ is close to  $\bar P.$ In the rest of the proof, for convenience, we will write $\arc(pq)$ to denote the length of the (shorter) arc with endpoints $p,q \in \partial B_1(0).$

\noindent{\textsc{ \color{blue}Case 1:}}  $\arc(Q\bar P)>\delta_1.$  Draw  the lines $l,m$ through the origin  and containing respectively two points $q',q''\in \partial B_1(0)$, lying in  the interior of the (shorter) arc $Q\bar P$ and satisfying $\arc(q'q'')=\delta_1$  (see Figure \ref{fig:circle}). 
	\begin{figure}[!ht]
\centering
\includegraphics[width=5.5cm, height=5.5cm]{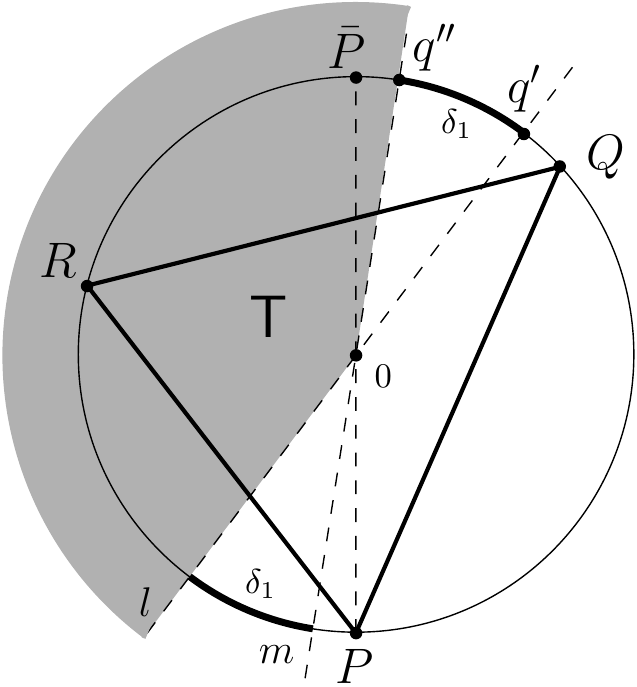}
\caption{The construction of $\T$ in  \textsc{Case 1}. In grey the region containing the point $R$, given by the intersection of the upper half planes determined by the lines $m$ and $l$}
\label{fig:circle}
\end{figure}
Then both $P,Q$  belong to  the lower open half planes determined by $l$ and $m$. Hence by Lemma \ref{lem:inscribed ball}  there must be a point $R\in \partial B_1(0)\cap \partial \Omega$ that lies in the intersection of closed upper half planes determined by $l$ and $m$. In particular (see Figure \ref{fig:circle}) $\arc(RP) \ge\delta_1,\arc(RQ)\ge\delta_1$. Moreover, since $Q$ belongs to the upper half plane $\{y\ge 0\}$, we have $\arc(PQ)\ge \delta_1.$ Hence the triangle ${\sf T}$  of vertices $P,Q,R$ satisfies $\alpha({\sf T})\ge \frac12{\delta_1}$. Indeed each of its angles is an angle at the circumference standing on an arc of length not less than $\delta_1$. Clearly $\T$  is inscribed in $\Omega$. Finally the area of $\T$ is
\begin{equation}\label{eq:area}
A(\T)=\frac{(\overline{PQ})^2 \sin \hat P  \sin \hat Q}{2\sin \hat R}\ge \sin^2\left(\frac14\arcsin( \frac{1}{3C_1})\right),
\end{equation}
where we used that $\overline{PQ}\ge \sqrt{2}$, since $Q$ lies in the upper half plane, and that $\alpha({\sf T})\ge \frac12{\delta_1}.$
This concludes the proof in the first case.

\noindent{\textsc{ \color{blue}Case 2:}} $\arc(Q\bar P)\le  \delta_1.$

	
	
	Consider the rectangle $(-A',A)\times [-\frac23,\frac23]\subset  \Omega$, where $A\ge 0$ and $A'\ge 0$ are taken to be  maximal with this property. By symmetry we can assume that $A'\ge A.$
 Note first that $A,A'>0$, since $B_1(0)\subset \Omega$. Moreover  by maximality the  side $\{A\}\times [-\frac23,\frac23]$ intersects $\partial \Omega.$ 
	
	If  $A > C_1$, then the rectangle $\mathcal R\coloneqq [-A'+\eps,A-\eps]\times [-\frac23,\frac23]$, for $\eps>0$ small enough, satisfies $ii)$ in the statement  and we are done. Hence we  assume from now on that $A\le  C_1.$ We also denote $\mathcal R_+\coloneqq [0,A]\times [-\frac23,\frac23].$ Let $r$ be the tangent line to circle $\partial B_1(0)$ containing $(A,\frac23)$ (the top-right corner of $\mathcal{R}$) having tangency point closer to $\bar P$ (see Figure \ref{fig:ball}).
 	\begin{figure}[!ht]
\centering
\includegraphics[width=9cm, height=9.5cm]{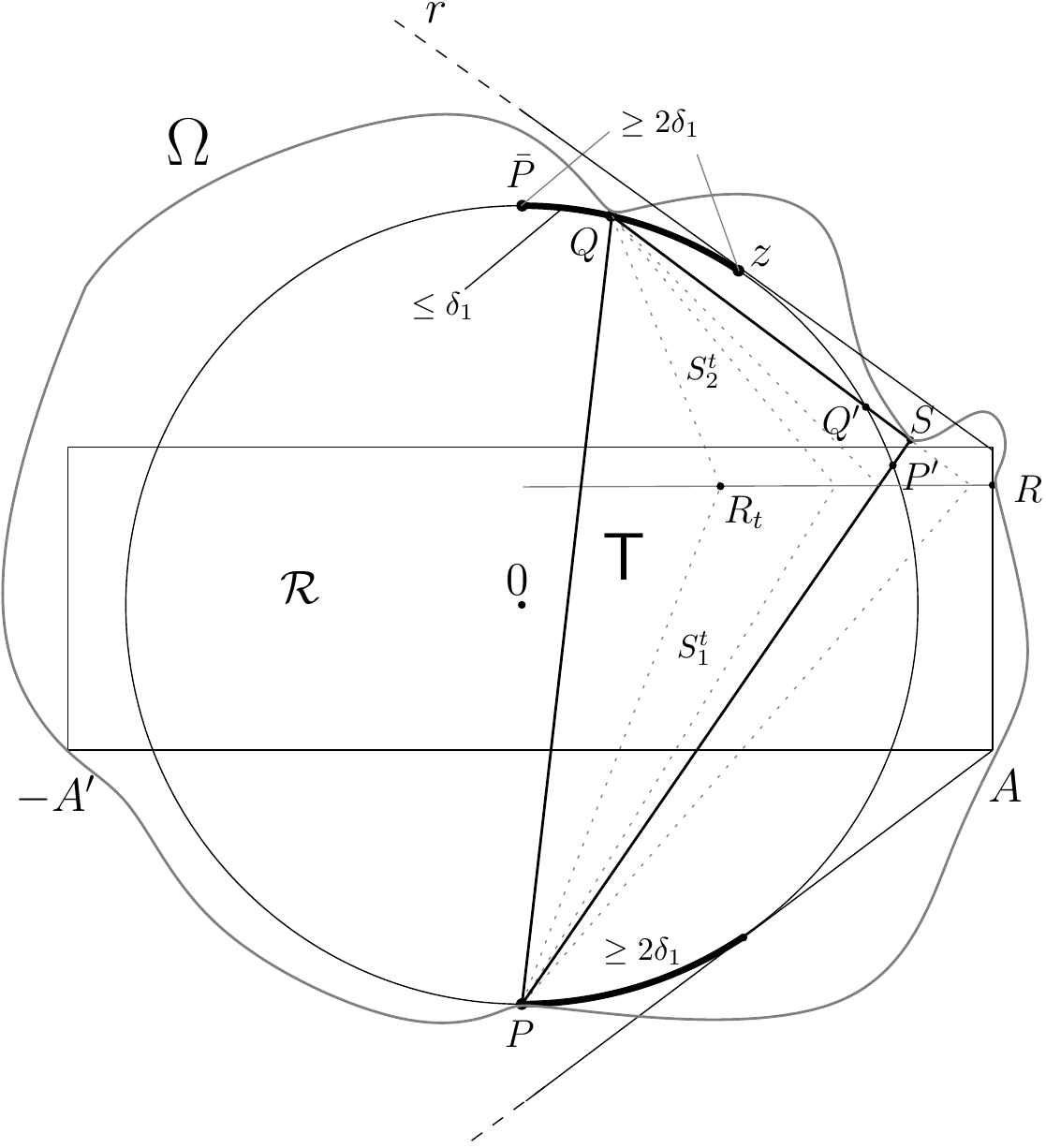}
\caption{The construction of $\T$ in \textsc{Case 2}. Note that $Q$ can also lie to the left of $P$, i.e.\ in the half space $\{x< 0\}$}
\label{fig:ball}
\end{figure}
 Denote by $z$ this tangency point.   Set $l\coloneqq \arc(z\bar P)$. Elementary trigonometric considerations show that
  \begin{equation}
      A\sin(l)={1-\frac23\cos(l)} \implies \sin(l)\ge \frac{1}{3A}\ge \frac{1}{3C_1}.
  \end{equation}
  In particular $l\ge \arcsin( \frac{1}{3C_1})$ and by the initial choice of $\delta_1$ it holds $l\ge 2\delta_1.$ Note that every segment having as endpoints $Q$ and a point in $\mathcal R_+\setminus B_1(0)$, passes below the point $z$ (i.e.\ intersects the vertical line containing $z$ at a point closer to $z$ than the $x$-axis). Indeed any such segment  must exit $B_1(0)$ and is contained in lower half plane determined by the tangent line $r$ (since so are its endpoints). This and recalling that $\arc(z\bar P)=l\ge 2\delta_1$ and $\arc(Q\bar P)\le \delta_1$ implies that
	\begin{equation}\label{eq:safe segments q}
		\parbox{11cm}{every segment having as endpoints $Q$ and a point in $\mathcal R_+\setminus B_1(0)$ intersects $\partial B_1(0)$ in a point $Q'\neq Q$ satisfying  $\arc({ QQ'})\ge \delta_1.$ }
	\end{equation}
	Symmetrically, since $P=(0,-1),$ we also get that
 \begin{equation}\label{eq:safe segments p}
		\parbox{11cm}{every segment having as endpoints $P$ and a point in $\mathcal R_+\setminus B_1(0)$ intersects $\partial B_1(0)$ in a point $P'\neq P$ satisfying  $\arc({ PP'})\ge 2\delta_1.$ }
	\end{equation}
 Recall that $\{A\}\times [-\frac23,\frac23]$ intersects $\partial \Omega$, hence there exists a point $R\coloneqq (A,R_y)\in \partial \Omega$ with $R_y\in [-\frac23,\frac23].$ For every $t\in [0,A]$ set $R_t\coloneqq (t,R_y)$ and consider the closed segments $S_1^t$, $S_2^t$ joining respectively $P$ and $Q$ with the point $R_t$ (see Figure \ref{fig:ball}). Set 
	$$t_0=\sup \{t \ :  \big(S^1_{t}\cup S^2_{t}\setminus \{Q,P\}\big)\cap\partial \Omega = \emptyset \}.$$
	Note that, since $R\in \partial \Omega$ and $B_1(0)\subset \Omega$, then $0< t_0\le A$. Therefore $R_{t_0}\in \mathcal R_+$. We claim that 
 \begin{equation}\label{eq:non empty closed}
      K\coloneqq \big((S^1_{t_0}\cup S^2_{t_0})\cap\partial \Omega\big) \setminus \{Q,P\}\text{ is a non-empty closed set.}
 \end{equation}
 Observe first that by \eqref{eq:safe segments p} and \eqref{eq:safe segments q} we have that
 \begin{equation}\label{eq:distance from pq}
     d(K,\{P,Q\})>0.
 \end{equation}
 This shows that $K$ is closed, since it is a closed set minus two isolated points. To show that $K$ is non-empty, by definition of $t_0$, note that there exist sequences $t_n \to t_0$ and $x_n \in  \big(S^1_{t_n}\cup S^2_{t_n}\setminus \{Q,P\}\big)\cap\partial \Omega $. Up to a subsequence we have that $x_n\to x\in  \big(S^1_{t_0}\cup S^2_{t_0}\big)\cap \partial \Omega$. However by \eqref{eq:distance from pq} we have  $d(x,\{P,Q\})>0$,  hence $x \in K$. This proves \eqref{eq:non empty closed}.
 Since $K$ is closed and non-empty, there exists $S\in K$ such that
 \begin{equation}\label{eq:S minimal}
     d(S,\{P,Q\})=d(K,\{P,Q\})>0.
 \end{equation}
  Consider the triangle $\T$ of vertices $P,Q,S \in \partial \Omega$.  By construction (the closure of) $\T$ is contained in
	\[
	\cup_{t \in[0,t_0]} S_1^{t}\cup S_2^t\subset \overline{\Omega}.
	\]
 Clearly the interior of the side $PQ$ is in $B_1(0)\subset \Omega$. Moreover both the sides $PS$ and $QS$ are contained in $S^1_{t_0}\cup S^2_{t_0}$. Hence, by definition of $S$, their interior can not intersect $K$ and so must be contained in  $\Omega$.
 This shows that $T$ is inscribed in $\Omega.$
	To give a lower bound on $\alpha(\T)$ we first observe that
 \begin{equation}\label{eq:fat two angles}
     \widehat{P},\widehat{Q}  \ge \frac12 \delta_1.
 \end{equation}
 To see this note that, since $S\in \rr^2\setminus B_1(0),$ the sets $QS\setminus\{Q\},PS\setminus\{P\}$ intersect   with $\partial B_1(0)$ at points that we denote by $Q',P'\in \partial B_1(0)$ respectively (see Figure \ref{fig:ball}). By \eqref{eq:safe segments q} and \eqref{eq:safe segments p} we have  that 
 $$\arc(P'Q)\ge\arc(QQ')\ge \delta_1, \quad \arc(Q'P)\ge \arc(P'P)\ge2\delta_1.$$ This shows \eqref{eq:fat two angles}, because $\widehat{P},\widehat{Q}$ are  angles at the circumference standing on arcs of length not less than $\delta_1.$
 Moreover, as $\arc(\bar PQ)\le \delta_1$, we also get
 \[
 \arc(P'Q)\le\arc(P'\bar P)+\arc(\bar PQ) =\pi-\arc(PP')+\arc(\bar PQ)\le \pi-\delta_1.
 \]
 Therefore $\widehat{P}\le \frac \pi2-\frac12\delta_1 $. Additionally $Q'\in \{x\ge 0\}$, hence  $\widehat{Q}\le \frac12\pi$, which implies
\[
\widehat{S}=\pi- \widehat{P}-\widehat{Q}\ge \frac12 \delta_1.
\]
 This shows that $\alpha(\T)\ge \frac 12\delta_1$. Finally, as in Case 1 we have $\overline{PQ}\ge \sqrt{2}$, so the lower bound \eqref{eq:area} for the area of $\T$ holds also in this case. This concludes the proof.
\end{proof}

It remains to prove Proposition \ref{prop:cases}. We give first an idea of the geometric construction. Recall that Proposition \ref{prop:cases} assumes the presence of a long rectangle $\mathcal R$ inside $\Omega$. Up to stretching the longer side, we can assume that $\mathcal R$ intersects $\partial \Omega$ in both its short sides.
The difficult case is when the only intersection points are precisely (or almost) two antipodal vertices of $\mathcal R$, that is the endpoints of a diagonal. In this case we say that $\mathcal R$ is \emph{bad rectangle}, otherwise we say that $\mathcal R$ is a \emph{good rectangle}. 
See for example Figure \ref{fig:idea} for an example of a bad rectangle $\mathcal R$ for which there are actually no triangles inscribed in $\Omega$ near $\mathcal R.$ 
In the case of a bad rectangle $\mathcal R$ the strategy will be to prolong half of $\mathcal R$ on both sides and to create two other rectangles $\mathcal R_1,\mathcal R_2$ inside $\Omega$. If one of them is good, we are done. If they are again both bad rectangles, this will force $\Omega$ to open up and will allow us to find a fatter rectangle $\tilde {\mathcal R},$ as in the conclusion of Proposition \ref{prop:cases}. We are not going to use the terminology bad and good rectangles in the proofs, its purpose  was only to give here an idea of the argument.

\begin{figure}[!ht]
\centering
\includegraphics[width=10cm, height=5cm]{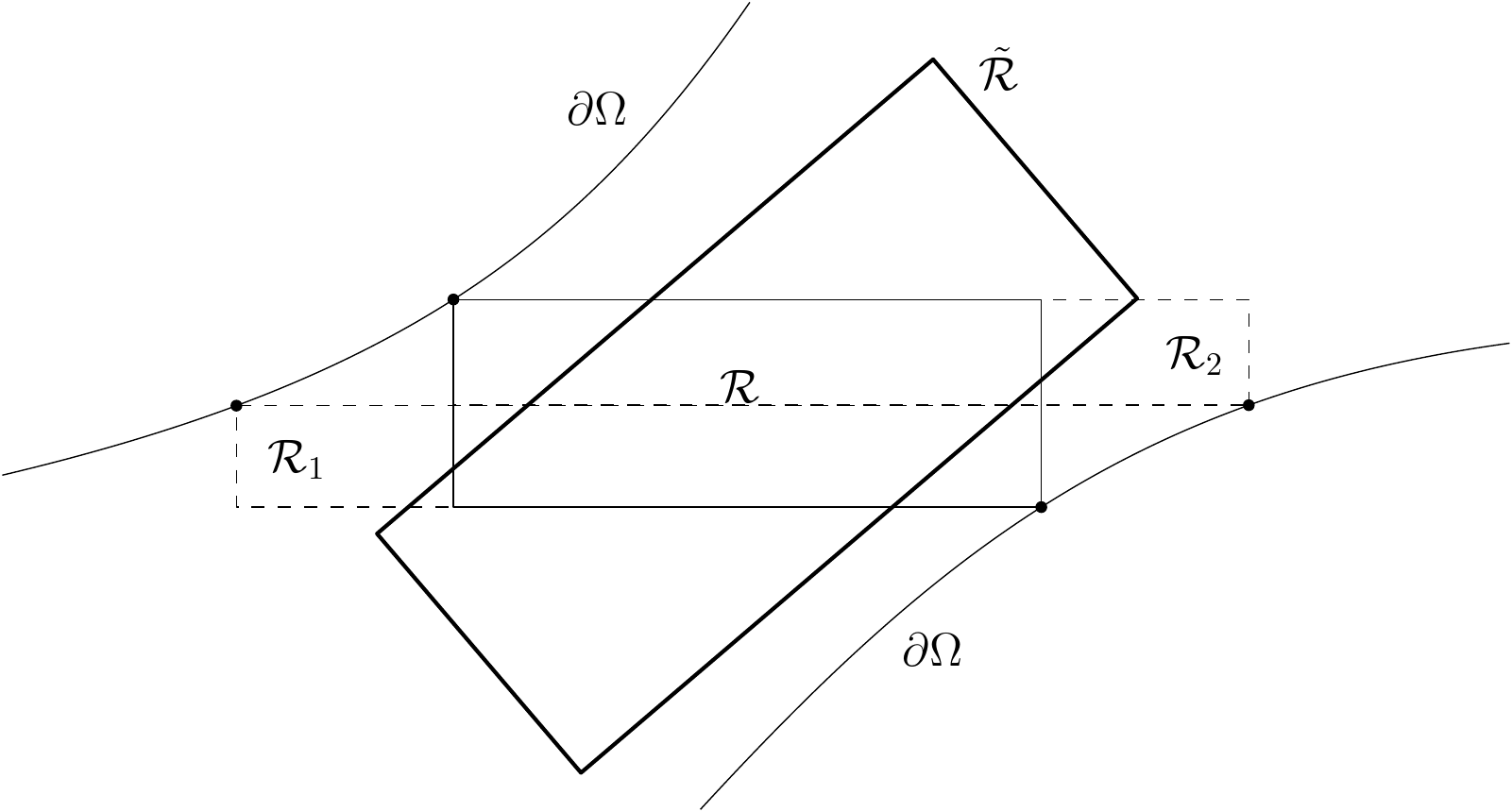}
\caption{Outline of the geometric idea for the proof of Proposition \ref{prop:cases}}
\label{fig:idea}
\end{figure}

We start with some preliminary technical results.
The main tool that we will use is  technical Lemma \ref{lem:trick set}, which roughly   says the following. Let $C$ be a closed set that contains two points $P,Q$ \emph{above} a  line $l$ and lies  \emph{below} $l$ in a vertical strip between $P,Q$. Suppose also that $C$ intersects every vertical line between $P$ and $Q$. Then we can find a segment with endpoints in $C$ which stays above $C$ (see Figure \ref{fig:lemma}). The key points is that this segment can be taken to be contained in arbitrarily thin vertical strip between $P,Q$. 
\begin{lemma}[Two points above-interval below]\label{lem:trick set}
    Let $C\subset \rr^2$ be a closed  set such that $C\subset \{(x,y)\ : \ y\le a\}$ for some $a \in \rr$ and let  $y=l(x)$ be a (non-vertical) line.   Suppose that:
    \begin{enumerate}[label=\roman*)]
        \item\label{it:above}  $C$ contains two points $(x_i,y_i)$, $i=0,1$, such that $y_i\ge l(x_i)$,  $i=0,1,$
        \item\label{it:vertical} $C\cap \{x=t\}\neq\emptyset$ for all $t\in[x_0,x_1]$,
        \item\label{it:below} there exists $[t_0,t_1]\subset [x_0,x_1]$ such that $C\cap\{t_0<x<t_1\}\subset l_\leq\coloneqq \{(x,y) \ : \ y\le l(x)\}$.
    \end{enumerate}
Then for every $\lambda\le |t_0-t_1|$ there exists $[s_0,s_1]\subset [x_0,x_1] $ with  $|s_0-s_1|\in [\lambda,2\lambda]$ and a line $y=\tilde l(x)$ such that  $C\cap\{s_0<x<s_1\}\subset \tilde l_\leq$ for every $t\in(s_0,s_1)$ and such that $(s_i,l(s_i))\in C$, $i=0,1.$ Moreover if the inequality in \ref{it:below} is strict, i.e.\  $C\cap\{t_0<x<t_1\}\subset l_<$,  we can also choose $t_0,t_1$ so that $C\cap\{s_0<x<s_1\}\subset \tilde l_<$.
\end{lemma} 

\begin{figure}[!ht]
\centering
\includegraphics[width=8cm, height=5.5cm]{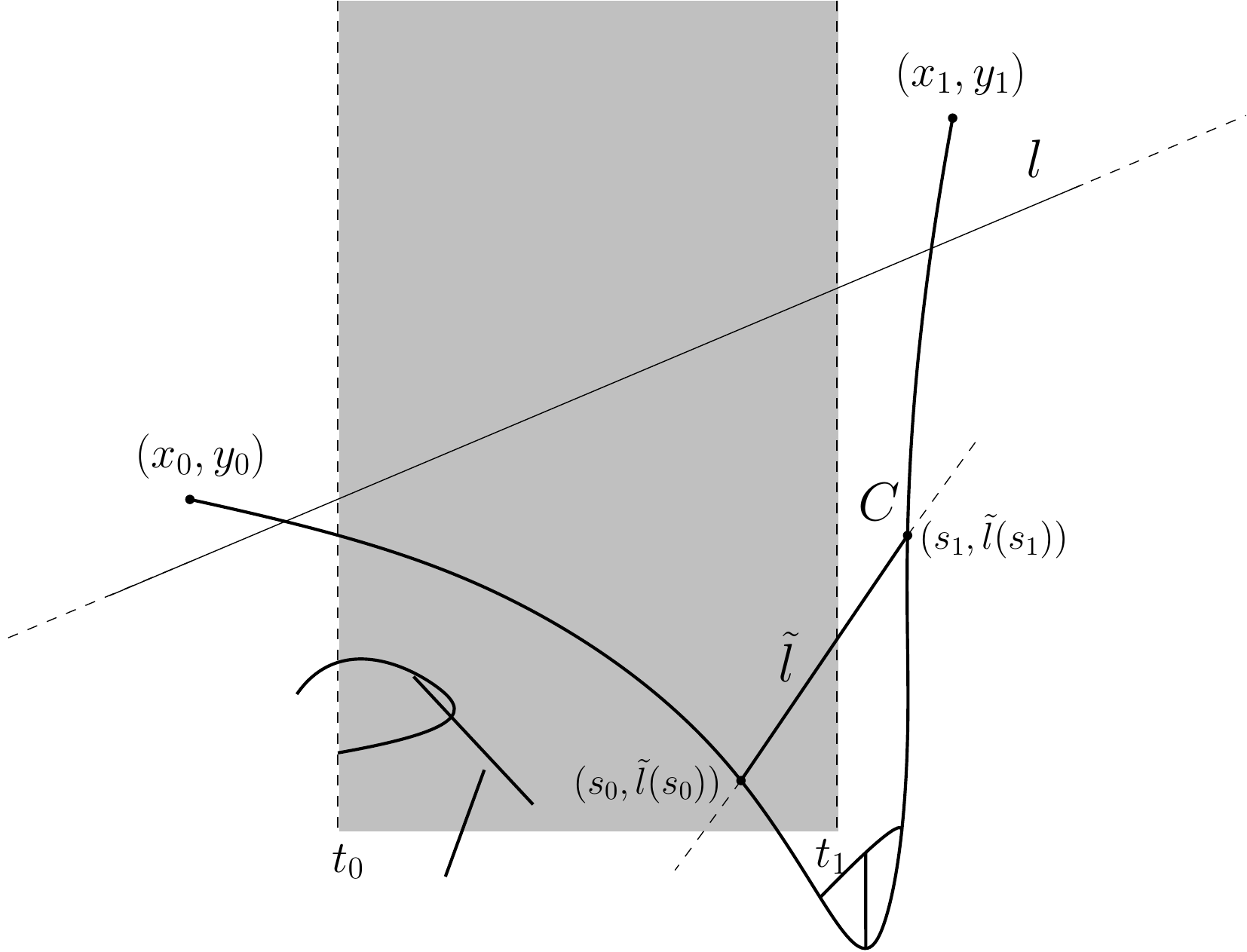}
\caption{Illustration of the statement of Lemma \ref{lem:trick set}. Note that the segment $[s_0,s_1]$ needs not to be contained in $[t_0,t_1]$}
\label{fig:lemma}
\end{figure}

The proof of Lemma \ref{lem:trick set} is postponed  to Section \ref{sec:lemma}, since it can be formulated as a stand-alone result about functions on the unit interval. 

The following lemma contains only an elementary computation which quantifies the fact that if if three points are not too far and not too close to one other, they form a fat triangle. 
\begin{lemma}\label{lem:box}
    Consider three points $P_0=(x_0,y_0)$, $P_1=(x_1,y_1)$, $P_2=(x_2,y_2)$ satisfying $x_0\le x_2\le x_1$, $y_0\le y_1\le y_2$ and
    \begin{equation}\label{eq:box bounds}
          3\le |x_0-x_1|\le 6, \quad  \frac13\le |y_2-y_1|\le 2,\quad \frac{|y_0-y_1|}{|x_0-x_1|}\le \frac{2}{|x_0-x_1|-2}.
    \end{equation}
    Then the triangle $T$ of vertices $P_1,P_2,P_3$ satisfies
    $\alpha(T)\ge \arctan\left(\frac{1}{47}\right).
    $
\end{lemma}
\begin{proof}
Observe first that the assumptions implies that $|y_0-y_1|\le  \frac{2|x_0-x_1|}{|x_0-x_1|-2}\le 12$ and $ \frac{|y_0-y_1|}{|x_0-x_1|}\le 2.$  Combining these we get 
\begin{equation}\label{eq:stupid}
    \frac{|y_1-y_0|(|y_1-y_0|+1/3)}{|x_0-x_1|}\le \frac{4|x_0-x_1|}{|x_0-x_1|-2}+\frac23,
\end{equation}
which we will use later on in the argument.
We will give separately a lower bound for each angle $\widehat P_i$, $i=1,2,3.$
    \begin{figure}[!ht]
\centering
\includegraphics[width=4cm, height=5cm]{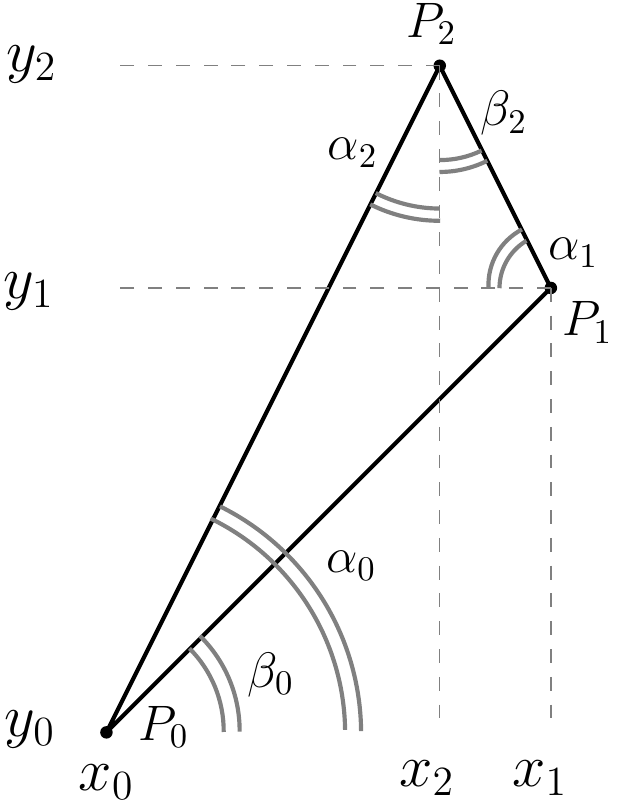}
\caption{The auxiliary angles used in the proof of Lemma \ref{lem:box}. Note that the relative position of the points $P_1,P_2,P_3$ is the one depicted, since by assumption $x_0\le x_2\le x_1$, $y_0\le y_1\le y_2$}
\label{fig:box}
\end{figure}
We have that $\widehat P_1\ge \alpha_1$, where $\alpha_1$ is the (smaller) angle formed by the segment $P_1P_2$ with the vertical axis (see Figure \ref{fig:box}). Then
\[
\widehat P_1\ge \alpha_1=\arctan\left( \frac{|y_2-y_1|}{|x_1-x_2|}\right)\ge \arctan\left( \frac{1}{18}\right).
\]
Similarly it holds that $\widehat P_2= \alpha_2+\beta_2$, where $\alpha_2$ and $\beta_2$ are the two angles into which $\widehat P_2$ is divided by the vertical line containing $P_2$ (see Figure \ref{fig:box}). In particular  $\alpha_2<\pi/2$ and $\beta_2<\pi/2$. Hence 
\begin{align*}
    \widehat P_2&= \alpha_2+\beta_2=\arctan\left( \frac{|x_0-x_2|}{|y_2-y_0|}\right)+\arctan\left( \frac{|x_2-x_1|}{|y_2-y_1|}\right)\\
    &\ge \arctan\left( \frac{|x_0-x_2|}{|y_2-y_0|}\right)+\arctan\left( \frac{|x_2-x_1|}{|y_2-y_0|}\right)\\
    &\ge \arctan\left( \frac{|x_0-x_1|}{|y_2-y_0|}\right)\ge \arctan\left(\frac{3}{|y_2-y_1|+|y_1-y_0|}\right)\ge 
    \arctan\left(\frac{3}{14}\right),
\end{align*}
where in the first inequality we used that $|y_2-y_0|\ge |y_2-y_1|$ and the monotonicity of $\arctan(\cdot)$,  in the second inequality the sub-additivity of $\arctan(\cdot)$ in $[0,\infty)$ and in the last step that $|y_0-y_2|\le12$ and $|y_0-y_1|\le 2.$ Finally we have that $\widehat P_0= \alpha_0-\beta_0$, where $\alpha_0$ and $\beta_0$ are the (smaller) angles  which the segments $P_0P_2$ and $P_0P_1$ form respectively with the horizontal axis (see Figure \ref{fig:box}). In particular $\alpha_0\le \pi/2,\beta_0\le \pi/2$ and $\hat P_0$.  Then
\begin{align*}
    \widehat P_0&= \alpha_0-\beta_0=\arctan\left( \frac{|y_2-y_0|}{|x_0-x_2|}\right)-\arctan\left( \frac{|y_1-y_0|}{|x_0-x_1|}\right)\\
    &\ge\arctan\left( \frac{|y_1-y_0|+1/3}{|x_0-x_1|}\right)-\arctan\left( \frac{|y_1-y_0|}{|x_0-x_1|}\right)\\
    &\ge \arctan\left( \frac{1/3}{|x_0-x_1|+\frac{|y_1-y_0|(|y_1-y_0|+1/3)}{|x_0-x_1|}}\right)
    \ge \arctan\left( \frac{1/3}{|x_0-x_1|+\frac{4|x_0-x_1|}{|x_0-x_1|-2}+\frac23}\right)\\
    &\ge \arctan\left( \frac{1/3}{15+\frac23}\right)=\arctan\left(\frac{1}{47}\right),
\end{align*}
where in the second line we used that $|x_0-x_2|\le |x_0-x_1|$, $|y_2-y_0|\ge |y_1-y_0|+1/3$ and monotonicity, in the third line we used the formula $\arctan(a)-\arctan(b)=\arctan\frac{a-b}{1+ab}$, valid for all $a,b\in \rr$ satisfying $\arctan(a)-\arctan(b)\in(-\frac{\pi}{2},\frac{\pi}{2})$, and also \eqref{eq:stupid}, while  in the last step that $t+\frac{4t}{t-2}\le 15$ for every $t\in[3,6]$ (recall that $|x_0-x_1|\in[3,6]$). Note that in the first line it might be that $|x_0-x_2|=0$, but the argument still works using the convention that $\arctan(+\infty)=\pi/2.$

Putting together all the lower estimates for $\hat P_1,\hat P_2,\hat P_3$ we conclude.
\end{proof}

The next result roughly says that, if we can find a segment with endpoints in $\partial \Omega$ and having a sufficiently large strip above it, i.e.\ a trapezoid, which is contained in $\Omega$ (see Figure \ref{fig:scheme}), then we can find a fat triangle inscribed in $\Omega.$ This will be used twice,  in the proof of the subsequent Lemma \ref{lem:easy case} and in the proof of Proposition \ref{prop:cases}.
\begin{lemma}[Trapezoid implies fat triangle]\label{lem:general box}
    Let $\Omega\subset \rr^2$ be open with $r(\Omega)=1$ and let $y=l(x)$ be a line (in Cartesian coordinates). Suppose that for some interval $[s_0,s_1]\subset \rr$, with $3\le |s_0-s_1|\le 6$, it holds $P_i=(s_i,l(s_i))\in \partial \Omega$, $i=0,1,$ and
    \begin{equation}\label{eq:trapez inside}
          \mathcal T\coloneqq \left \{(x,y) \ : \  x\in (s_0,s_1),\quad  l(x)\le y\le  c \right\}\subset  \Omega,
    \end{equation}
    for some $c\ge l(s_i)+\frac13$, $i=0,1.$
    Then $\Omega$ contains an inscribed triangle $\T$ with $\alpha(\T)\ge \arctan\left(\frac{1}{47}\right)$ and $A(\T)\ge \frac{9}{4420}.$
\end{lemma}
\begin{proof}
  Set $d\coloneqq \sup\{ t\ge c \ : \  \{t \} \times [s_0+\frac13,s_1-\frac13]\subset  \Omega\}$ (this set is non-empty because by \eqref{eq:trapez inside} it holds $\{c \} \times (s_0,s_1)\subset \Omega$). This supremum is finite because $\Omega$ is bounded. By maximality  there exists a point $Q=(x_Q,d)\in \partial \Omega\cap ([s_0+\frac13,s_1-\frac13]\times \{d\})$. Consider the closed triangle $\T_0$ of vertices $P_0,P_1,Q$ (see Figure \ref{fig:scheme}). Then by \eqref{eq:trapez inside} and since $x_Q\in (s_0,s_1)$ we have
  \begin{equation}\label{eq:T0 below c}
      \T_0\cap \partial \Omega\cap \{y < c\}=\{P_0,P_1\},
  \end{equation}
  indeed $\T_0\subset \{y\ge l(x)\}$, because $P_0,P_1,Q\in \{y\ge l(x)\}$. Define the closed set
  \[
  K\coloneqq    \T_0\cap \partial \Omega\cap \{y \ge c\},
  \]
 which is non-empty because $Q\in K.$ Then there exists $P_2=(x_2,y_2)\in K$ which minimizes the $y$-coordinate in $K$, i.e.\
 \begin{equation}\label{eq:minimum y}
      y_2=\min\{y \ : \ (x,y)\in K\}\in [c,d].
 \end{equation}
 Consider now the triangle $\T$ of vertices $P_0,P_1,P_2$ (see Figure \ref{fig:scheme}). In particular $\T\subset \T_0,$ because $P_0,P_1,P_2 \in \T_0.$ Therefore  we have
 \begin{equation}\label{eq:T above}
      \T\cap \partial \Omega \cap \{y<c\}\subset  \T_0\cap \partial \Omega \cap \{y<c\} \overset{\eqref{eq:T0 below c}}{=}\{P_0,P_1\}. 
 \end{equation}
 Moreover 
 \begin{equation}\label{eq:T below}
      (\T\setminus \{P_2\})\cap \partial \Omega \cap \{y\ge c\}\subset  \T_0\cap \partial \Omega \cap \{c\le y<y_2\} =K\cap \{y<y_2\}=\emptyset,
 \end{equation}
 where the first inclusion follows observing that $P_0,P_1\in \{y< y_2\}$, $P_2 \in \{y=y_2\}$ and by convexity, while the  last equality follows from \eqref{eq:minimum y}. Combining \eqref{eq:T below} and \eqref{eq:T above} we obtain that $\T$ is inscribed in $\Omega.$ Indeed $\T$ intersects $\partial \Omega$ exactly at the vertices and also intersects $\Omega$, therefore by convexity  $\T$ minus its vertices must be contained in $\Omega.$
  \begin{figure}[!ht]
\centering
    \includegraphics[width=6cm, height=6cm]{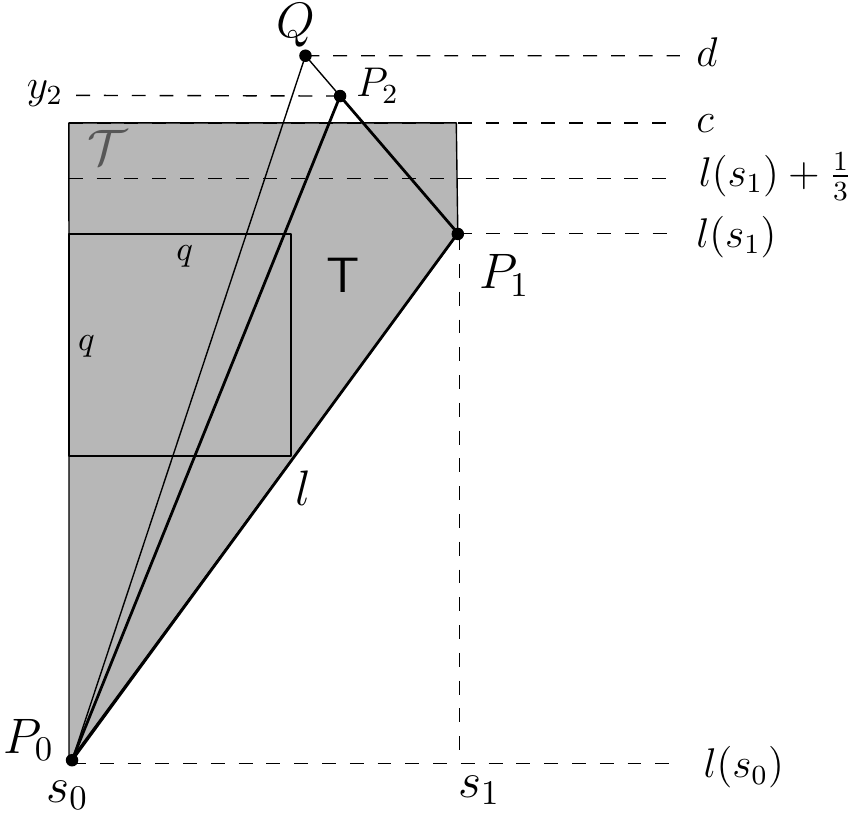}
\caption{Illustration of the construction in the proof of Lemma \ref{lem:general box}. The trapezoid region $\mathcal T$ defined in the statement is depicted in grey}
\label{fig:scheme}
\end{figure}
It remains to prove that $\alpha(\T)\ge \arctan\left(\frac{1}{47}\right)$ and the lower bound on the area $A(\T).$ We aim to apply Lemma \ref{lem:box} to the points $P_i\coloneqq (x_i,y_i)\coloneqq (s_i,l(s_i))$, $i=0,1$  and $P_2=(x_2,y_2)$. By symmetry we can assume that $l(s_0)\le l(s_1).$ Therefore by construction we have that $x_0\le x_2\le x_1 $ and $y_0\le y_1\le y_2\le d$, hence it remains only to check the assumptions in \eqref{eq:box bounds}. The first in \eqref{eq:box bounds} is true because by assumption $3\le |s_0-s_1|\le 6$. Next note that  by definition of $d$ and by assumption \eqref{eq:trapez inside} the rectangle $(s_0+\frac13,s_1-\frac13)\times (l(s_1),d)$ is contained in   $\Omega$. Hence it must hold $ d\le l(s_1)+2$, otherwise this rectangle would have both sides strictly greater than 2  (recall that $|s_0-s_1|\ge 3)$, which would contradict $r(\Omega)=1$. Moreover by assumption $l(s_1)\le c-\frac13\le y_2-\frac13$. Combining the last two observation we get $1/3\le |y_2-l(s_1)|\le 2$, which proves the second in \eqref{eq:box bounds}. For the last in \eqref{eq:box bounds} we need to verify that $\frac{|l(s_1)-l(s_0)|}{|s_0-s_1|}\le \frac{2}{|s_0-s_1|-2}$, $i=0,1.$ The idea is that if $|l(s_1)-l(s_0)|$ is too big, then $P_0$ is much below $P_1$ and  we can fit inside $\Omega$ a square of side strictly bigger than 2  which would contradict $r(\Omega)=1$ (see Figure \ref{fig:scheme}). More precisely observe that the slopes of $l$ coincides with $m\coloneqq \frac{|l(s_1)-l(s_0)|}{|s_0-s_1|}$. Then by elementary geometric considerations we deduce that the trapezoid $\mathcal T\subset \Omega$, defined in the statement, contains an (open) square of side $q$, where $q$ satisfies $m=\frac{q}{|s_0-s_1|-q}$ (see Figure \ref{fig:scheme}). Since $r(\Omega)=1$ it must be $q\le 2$ and so
  \[
  \frac{|l(s_1)-l(s_0)|}{|s_0-s_1|}=m=\frac{q}{|s_0-s_1|-q}\le \frac{2}{|s_0-s_1|-2},
  \]
which is what we wanted. Finally, since $\overline{P_0P_1}\ge |s_0-s_1|\ge 3$ and $\alpha(T)\ge \arctan\left(\frac{1}{47}\right)$, we have the following lower bound for the area of $\T$
\[
A(\T)=\frac{(\overline{P_0P_1})^2 \sin \hat P_0  \sin \hat P_1}{2\sin \hat P_2}\ge \frac92 \sin\left(\arctan\left(\frac{1}{47}\right)\right)^2=\frac{9}{4420}.
\]
The proof is concluded.
\end{proof}

The next and last preliminary result says  that if a rectangle contained in $\cl( \Omega)$ intersects $\partial \Omega$ at both its short sides and at two points, which are not almost the endpoints of a diagonal, then we can find a sufficiently fat triangle inscribed in $\Omega$ (in a  quantified way). The proof will make use of both  Lemma \ref{lem:trick set} and Lemma \ref{lem:general box}.
\begin{lemma}[Good rectangle implies fat triangle]\label{lem:easy case}
	Let $\Omega\subset \rr^2$ be an open and bounded set with $r(\Omega)=1$. Suppose that $(-A,A)\times[-a,a] \subset \Omega$ with
	$A\ge 2$ and $a\ge 1/3$. Suppose also that $\partial \Omega \cap \{A\}\times [-a,a-1/3]\neq \emptyset$ and $\partial \Omega \cap \{-A\}\times [-a,a-1/3]\neq \emptyset$. Then $\Omega$ has an inscribed triangle $\T$ with $\alpha(\T)\ge \arctan\left(\frac{1}{47}\right) $ and $A(\T)\ge \frac{9}{4420}.$
\end{lemma}
\begin{proof}
We aim to apply first Lemma \ref{lem:trick set} and then Lemma \ref{lem:general box}.
Consider the closed set 
\[
C\coloneqq [-A,A] \times (-\infty,a-1/3]\cap \partial \Omega.
\]
Consider the horizontal line  $l\coloneqq \{y=-a\}$. By assumptions $C$ contains two points $(-A,y_0)$ and $(A,y_1)$ with $y_0,y_1 \in [-a,a-1/3]$. Moreover, since $(-A,A)\times \{-a\} \subset \Omega$ and $\Omega$ is bounded, $C$ must intersect every vertical line $\{x=t\}$ , $t \in (-A,A)$. Indeed for every $t \in (-A,A)$ we can take $(t,-a)\in \Omega$ and decrease the second coordinate until we hit $\partial \Omega$ in a point which must lie in $C.$  
Finally the assumption  $(-A,A)\times[-a,a] \subset \Omega$ also implies that the set $C$ lies strictly below $l$ in the strip $\{-A<x<A\}$, that is $C\cap \{-A<x<A\}\subset \{y<-a\}$. Hence applying Lemma \ref{lem:trick set} taking $\lambda=3$ (which is admissible since $2A\ge 3$) we obtain $s_0,s_1 \in [-A,A]$ with $3\le |s_0-s_1|\le 6$ and a line $y=\tilde l (x)$ satisfying $ P_i\coloneqq (s_i,\tilde l(s_i))\in C\subset \partial \Omega$, $i=0,1,$ and such that $C$ lies strictly below $\tilde l$ in $(s_0,s_1)$,  i.e.\ 
\begin{equation}\label{eq:C strictly below}
    \{ (x,y) \in C \ : \ x\in (s_0,s_1),\quad y\ge  \tilde l(x)\}=\emptyset.
\end{equation}
We stress that it could very well be that  $s_0=-A$ or $s_1=A.$
 We claim that
\begin{equation}\label{eq:T inside in easy case}
     \mathcal T\coloneqq \{(x,y) \ : \ x \in (s_0,s_1), \, \tilde l(x)\le y \le a \}\subset \Omega.
\end{equation}
If this was true an application of Lemma \ref{lem:general box} would yield the result. Indeed all the hypotheses would be satisfied   noting that, by definition of $C$ and since $(s_i,\tilde l(s_i))\in C,$ we have  $a-\frac13\ge \tilde l(s_i)$ for $i=0,1.$ 

We prove \eqref{eq:T inside in easy case} by contradiction. Suppose  that  $\mathcal T\cap(\rr^2\setminus  \Omega) \neq \emptyset$ and let $P=(x_P,y_P)\in \mathcal T\cap(\rr^2\setminus  \Omega) $. Since $x_P \in (-A,A)$,  the assumptions yield that $\{x_P\}\times [-a,a]\subset \Omega,$ which gives that $(x_p,-a)\in \Omega$ and that $y_P<-a$ (indeed by definition of $\mathcal T$ we have $y_P\le a$). Therefore the vertical segment $\{x_P\}\times [y_P,-a]$ intersects $\partial \Omega$, because it has one endpoint in $\Omega$ and one in $\rr^2\setminus\Omega$.  Hence, being  contained in $ [-A,A] \times (-\infty,a-1/3],$ it  also intersects  $C$. However by definition of $\mathcal T$ it holds $y_P\ge \tilde l(x_P)$, which contradicts \eqref{eq:C strictly below} and concludes the proof.
\end{proof}

We can finally prove the second  main proposition that was used to deduce Theorem \ref{thm:main}.
\begin{proof}[Proof of Proposition  \ref{prop:cases}]
The proof will be divided in several cases. In all except the last one we will find an inscribed triangle $\T$ satisfying $i)$ of the statement, while in the very last case we will find a rectangle $\tilde {\mathcal R}\subset  \Omega$ as in point $ii)$.

By hypotheses there exists a closed rectangle $\mathcal R\subset \Omega$ of sides-length $2a\times 2A$, with $a\in [2/3,1]$ and $A\ge 13.$ 
	Up to rotating and translating $\Omega$ we can assume that  ${\mathcal R}=[-A,A]\times[-a,a].$ Moreover, up to enlarging $A$ we can assume that 
 \begin{equation}\label{eq:R new assumptions}
      (-A,A)\times [-a,a]\subset \Omega, \quad \{-A\}\times [-a,a]\cap \partial \Omega \neq \emptyset, \quad \{A\}\times [-a,a]\cap \partial \Omega \neq \emptyset,
 \end{equation}
 but giving up the property that $\mathcal R\subset \Omega$ (which can not, and never will,  be used again in the rest of the proof). Indeed  we can replace $\mathcal{R}$  with the rectangle $[-B',B]\times[-a,a],$ where $B\coloneqq \sup \{ t \ : \ [-A,t]\times[-a,a]\subset \Omega\}$ and $B'\coloneqq \sup \{ t \ : \ [-t,A]\times[-a,a]\subset  \Omega\}$ (and then perform another translation of $\Omega$ to obtain a rectangle centred at the origin).

\noindent {\color{blue} \sc Case 1a}: Suppose that 
\begin{equation}\label{eq:c1a}
	\partial \Omega \cap \{-A\}\times [-a,a/2]\neq \emptyset, \quad \partial \Omega \cap \{A\}\times [-a,a/2] \neq \emptyset.
\end{equation}
Since $a/2\ge 1/3 $, $A\ge 2 $ and by the first in \eqref{eq:R new assumptions}, we can apply Lemma \ref{lem:easy case} to the rectangle $\mathcal R$ and deduce that  $\Omega$ has an inscribed triangle $\T$ with $\alpha(\T)\ge \arctan\left(\frac{1}{47}\right)$ and area $A(\T)\ge \frac{9}{4420}.$
	
\noindent {\color{blue} \sc Case 1b}: Suppose that 
\begin{equation}\label{eq:c1b}
	\partial \Omega \cap \{-A\}\times [-a/2,a]\neq \emptyset, \quad \partial \Omega \cap \{A\}\times [-a/2,a] \neq \emptyset.
\end{equation}
Up to changing the sign of the $y$-coordinate, this is the same as Case 1A.

\medskip
	
From now on we can assume that neither \textsc{Case 1.A} nor \textsc{Case 1.B} applies, i.e.\ we can assume that at least one of the intersections in \eqref{eq:c1a} and \eqref{eq:c1b} is empty. This means that either the first intersection in \eqref{eq:c1a} and the second in \eqref{eq:c1b} are empty, or vice versa (see Figure \ref{fig:case1.5}). Indeed it is not possible for only the first or only the second intersection in both \eqref{eq:c1a} and \eqref{eq:c1b} to be empty because their respective unions is non-empty thanks to \eqref{eq:R new assumptions}. Up to changing the sign of the $x$-coordinate we can assume from now on that:
\begin{equation*}
    \parbox{12cm}{
the first in \eqref{eq:c1a} and the second in \eqref{eq:c1b} fail, which  implies (by \eqref{eq:R new assumptions}):
\begin{equation}\label{eq:c2}
	\partial \Omega \cap \{-A\}\times [a/2,a]\neq \emptyset, \quad \partial \Omega \cap \{A\}\times [-a,-a/2]\neq \emptyset.
\end{equation}
}
\end{equation*}
Note that this assumption (and in particular \eqref{eq:c2}, which is the property that will be actually used in the sequel) remains true if $\Omega$ is rotated of $\pi$-degrees around the origin.

\begin{figure}[!ht]
\centering
\includegraphics[width=14cm, height=3cm]{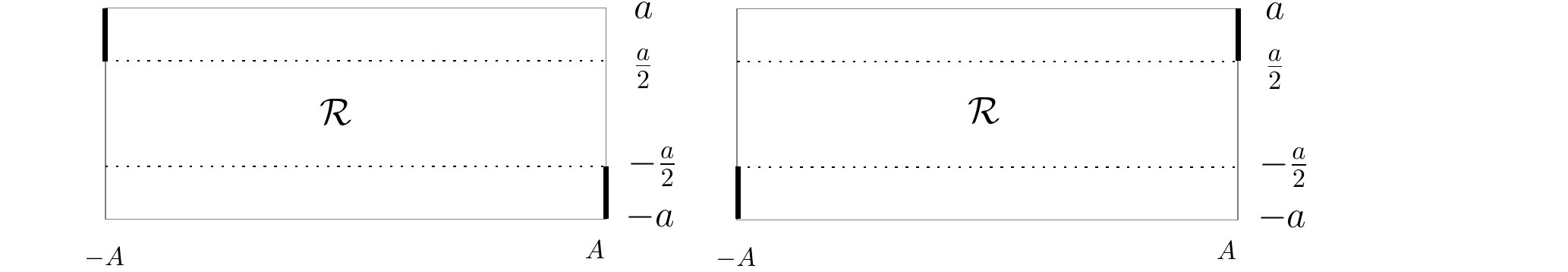}
\caption{Illustration of the two scenarios not covered by \textsc{Case 1.A} or \textsc{Case 1.B}. Note that they are the same, up to changing the sign of the $x$-coordinate. The segments highlighted with thicker lines intersect $\partial\Omega$. The image on the left corresponds to  \eqref{eq:c2}. Confront this picture also with the previous Figure \ref{fig:idea}}
\label{fig:case1.5}
\end{figure}
 Consider  the rectangles 
\[
\mathcal{R}_1\coloneqq  [-A_1,A] \times [-a,0], \quad  \mathcal{R}_2\coloneqq  [A,A_2] \times [0,a],
\]
where $A_i$, $i=1,2$ are defined by
\[
A_1\coloneqq \sup \{ t\ \ : \ [-t,A] \times [-a,0] \subset \Omega\}, \quad A_2\coloneqq \sup \{ t\ \ : \ [-A,t] \times [0,a] \subset \Omega\}
\]
(see Figure \ref{fig:case2}). Recall that $[-A,0]\times [-a,0]\subset \Omega$ and $[0,A]\times [-a,0]\subset \Omega$, by the first in \eqref{eq:R new assumptions} and because we are assuming that  the first in \eqref{eq:c1a} and the second in \eqref{eq:c1b} fail. In particular $A_i$, $i=1,2$ are both finite and   $A_i> A,$ $i=1,2$.
In particular $\mathcal{R}_i\subset \cl({\Omega})$ and by maximality
\begin{equation}\label{eq:sideinter}
	\{-A_1\}\times[-a,0]\cap \partial \Omega \neq \emptyset,\quad 	\{A_2\}\times[0,a]\cap \partial \Omega \neq \emptyset.
\end{equation}

\noindent {\color{blue} \sc Case 2}:  Suppose that at least one of the following holds
\begin{equation}\label{eq:case2}
\{-A_1\}\times[-a,-a/2]\cap \partial \Omega\neq \emptyset,\quad  	\{A_2\}\times[a/2,a]\cap \partial \Omega\neq \emptyset.
\end{equation}
 Up to rotate $\Omega$ of $\pi$-degrees around the origin (this is allowed because the current assumption \eqref{eq:c2} remains true after this transformation) and exchanging the values of $A_1$ and $A_2,$ we can assume $\{-A_1\}\times[-a,-a/2]\cap \partial \Omega\neq \emptyset$ (see Figure \ref{fig:case2}). Moreover, by the second in \eqref{eq:c2} we
have $\{A\}\times [-a,-a/2]\cap  \partial \Omega \neq \emptyset$. Applying Lemma \ref{lem:easy case} to the rectangle $\mathcal R_1$ (again noticing that $a/2\ge 1/3 $) we deduce that $\Omega$ has an inscribed triangle $\T$ with $\alpha(\T)\ge \arctan\left(\frac{1}{47}\right)$ amd $A(\T)\ge \frac{9}{4420}.$

\begin{figure}[!ht]
\centering
\includegraphics[width=14cm, height=3.5cm]{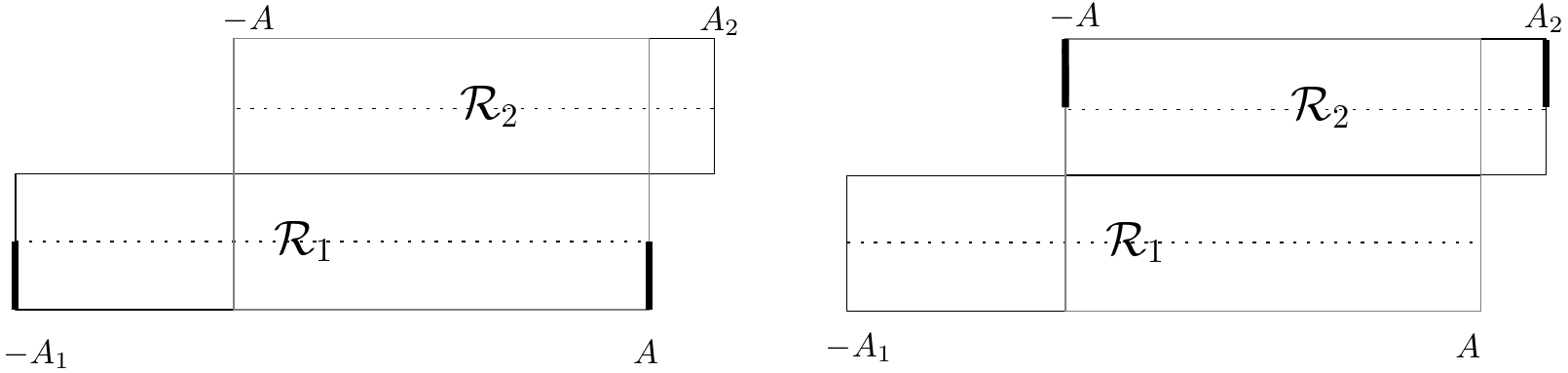}
\caption{Illustration of the two scenarios in \textsc{Case 2}. The segments highlighted with thicker lines intersect  $\partial\Omega$. Note that the two pictures are the same, up to a rotation of $\pi$-degrees around the origin and exchanging the values of $A_1,A_2$}
\label{fig:case2}
\end{figure}

\noindent {\color{blue} \sc Case 3}:  Suppose, contrary to \eqref{eq:case2}, that
\begin{equation}\label{eq:c3}
	\{-A_1\}\times[-a,-a/2]\cap \partial \Omega=\emptyset, \quad 	\{A_2\}\times[a/2,a]\cap \partial \Omega=\emptyset.
\end{equation}
This, recalling \eqref{eq:sideinter}, implies that 
\begin{equation}\label{eq:c3bis}
 \{-A_1\}\times[-a/2,0]\cap \partial \Omega\neq \emptyset  , \quad  \{A_2\}\times[0,a/2]\cap \partial \Omega\neq \emptyset
\end{equation}
(see Figure \ref{fig:case3}).
\begin{figure}[!ht]
\centering
\includegraphics[width=8.5cm, height=3.5cm]{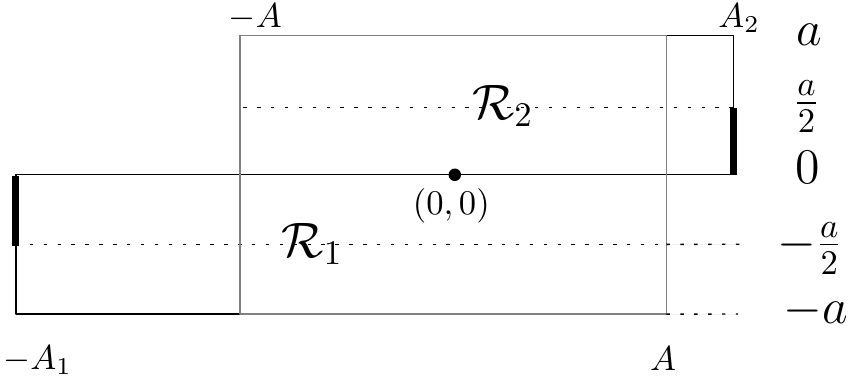}
\caption{The situation  in \textsc{Case 3}. The segments highlighted with thicker lines intersect  $\partial\Omega$. Observe that, up to swapping the values of $A_1$ and $A_2$, the picture is unchanged by a rotation of $\pi$-degrees around the origin}
\label{fig:case3}
\end{figure}

Let $y=l(x)$, $y=m(x)$ be the two lines (in Cartesian coordinates) respectively satisfying $l(A_2)=0$, $l(A-3)=-a$ and $m(-A_1)=0$, $m(-A+3)=a,$ that is:
\begin{equation}\label{eq:equations l and m}
    l(x)\coloneqq \frac{a(x-A_2)}{A_2-A+3}, \quad m(x)\coloneqq \frac{a(x+A_1)}{A_1-A+3}.
\end{equation}
Let $T_l,T_m$ be the (closed) triangles enclosed respectively by the lines $l$, $\{y=-a\}$, $\{x=-A_1\}$ and $m$, $\{y=a\}$, $\{x=A_2\}$ (see Figure \ref{fig:case3bis}). Note that, since $A\ge 2$, we have both $A-3>-A$ and $-A+3< A$ hence both  the lines $l$ and $m$ intersect respectively the interior of the bottom and top horizontal sides of $\mathcal R$ (as illustrated in Figure \ref{fig:case3bis}) 

\begin{figure}[!ht]
\centering
\includegraphics[width=11cm, height=5cm]{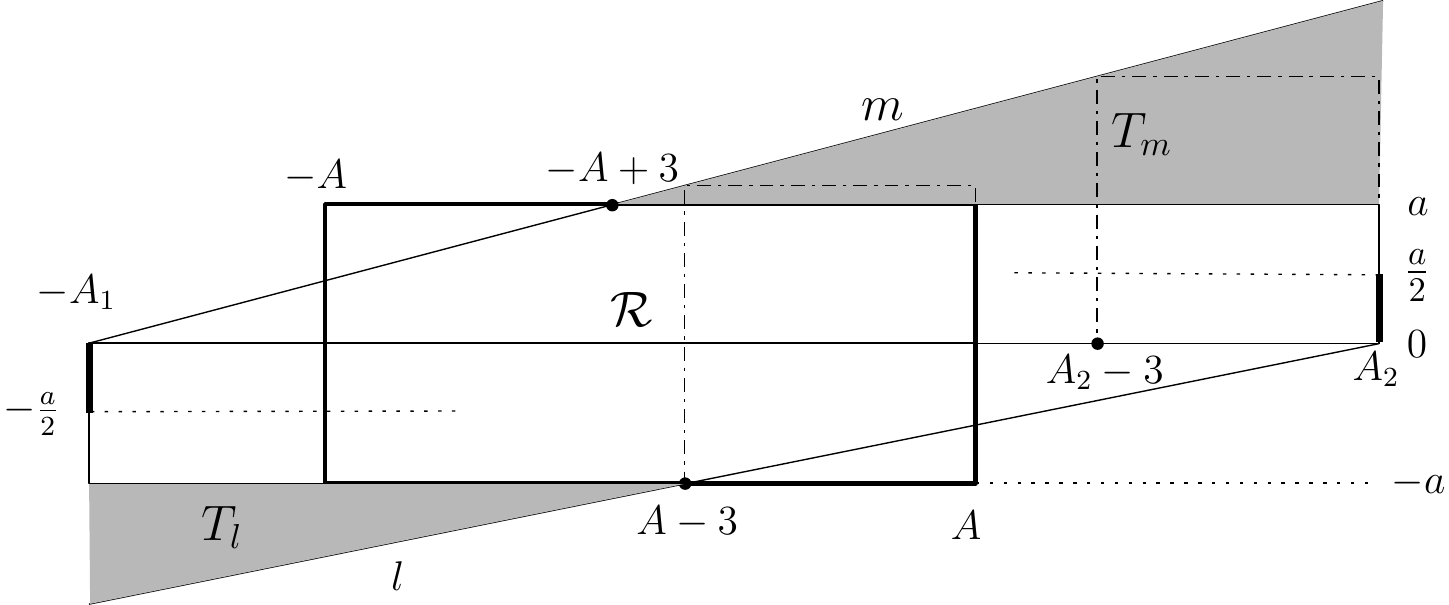}
\caption{If at least one of the triangles $T_l, T_m$ intersects $\partial \Omega$ we are in \textsc{Case 3a}, otherwise  in \textsc{Case 3b}. Note that if $T_m$ does not intersect $\partial \Omega$, then $A_1$ can not be much smaller than $A_2$ and $A$, otherwise we can fit large rectangles inside $\Omega$  as in the picture, contradicting $r(\Omega)=1$ (the same  for $T_l$ and $A_2$). This observation will be  used in the proof of \eqref{eq:superbad}}
\label{fig:case3bis}
\end{figure}
We distinguish two subcases.

\noindent {\color{blue} \sc Case 3a}: at least one of the triangles $T_l, T_m$ intersects $\partial \Omega$. Recall that the assumption of \textsc{Case 3} are preserved by a rotation of $\pi$-degrees around the origin and up to swapping the values of $A_1$ and $A_2$  (see Figure \ref{fig:case3}). Observe also that a rotation of $\pi$ degrees maps the line $m$ to $l$ (and vice-versa) after swapping the values of $A_1$ and $A_2$. By these observations and since no assumption on the values of $A_1$ and $A_2$ has been made yet, up to a rotation  we can then assume that 
\begin{equation}\label{eq:c3a}
    T_l\cap \partial \Omega\neq \emptyset.
\end{equation}
We aim to apply Lemma \ref{lem:trick set} with the line $l$ and with the closed set $C\coloneqq \partial \Omega \cap \mathcal H$, where
\begin{equation}\label{eq:def H}
    \mathcal H\coloneqq [-A_1,0]\times(-\infty,-\frac a2]\cup [0,A_2]\times(-\infty,\frac a2]
\end{equation}
(see Figure \ref{fig:case3.a}). 
\begin{figure}[!ht]
\centering
\includegraphics[width=12cm, height=6cm]{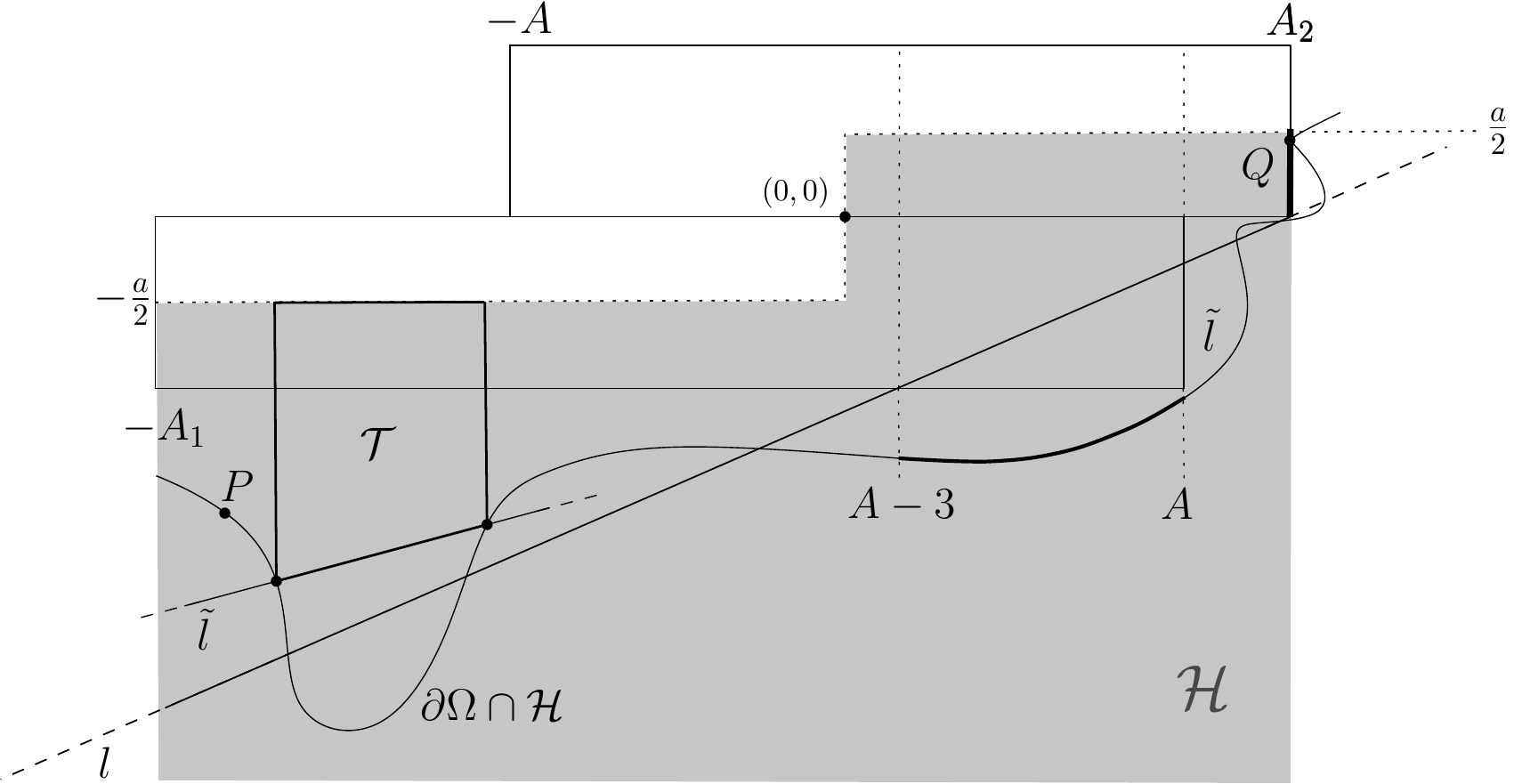}
\caption{The construction in \textsc{Case 3.a}}
\label{fig:case3.a}
\end{figure}
We now check that the hypotheses of Lemma \ref{lem:trick set} are satisfied. By the second in \eqref{eq:c3bis} we have that $C$ contains a point $Q=(A_2,y_1)$ satisfying 
$$a/2\ge y_1\ge 0=l(A_2)$$ (see Figure \ref{fig:case3.a}). Moreover by \eqref{eq:c3a} $C$ also contains a point $P=(x_0,y_0)$ with $x_0\le A-3$ and with $l(x_0)\le y_0\le -a$ (see Figure \ref{fig:case3.a}). Hence assumption $\ref{it:above}$ of Lemma \ref{lem:trick set} is verified. On the other hand $C\cap \{A-3< x<A\}$ is contained in $l_<$ (i.e.\ $C$ lies strictly below $l$), indeed $l\restr{[A-3,A]}\ge -a$, while $C\cap \{A-3<x< A\}\subset \{y< -a\}$ because $(A-3,A)\times [-a,a]\subset  \Omega $ by \eqref{eq:R new assumptions}. Hence assumption $\ref{it:below}$ of Lemma \ref{lem:trick set} is also true. It remains to check assumption $\ref{it:vertical}$, for which it suffices to show  that $C$ intersects every vertical line $\{x=t\}$, $t \in (-A_1,A_2)$. By construction $(-A_1,0)\times\{-a/2\}\subset \Omega$ and $[0,A_2)\times\{a/2\}\subset {\Omega}$. Hence, since $\Omega$ is bounded, for every $x \in (-A_1,A_2)$ there exists at least one $y \in \rr $ with $(x,y)\in \partial \Omega \cap C$ (indeed we can take $(x,y')\in \Omega$ with $y'=a/2$ (or $-a/2$) and decrease continuously $y'$ until we hit $\partial \Omega$). Therefore we apply Lemma \ref{lem:trick set} (with $\lambda=3$) and obtain a line $y=\tilde l(x)$ and an interval $[s_0,s_1]\subset [-A_1,A_2]$ with $3\le |s_0-s_1|\le 6$ with  $(s_i,\tilde l(s_i))\in C\subset \partial \Omega$, $i=0,1,$ and such that $C$ lies strictly below $\tilde l$ in $(s_0,s_1)$, i.e.\  
\begin{equation}\label{eq:C below tilde l}
    C\cap\{s_0< x< s_1\}\subset \tilde l_{<}=\{y< \tilde l(x)\}.
\end{equation}
We stress that it could be that $s_0=-A_1$ or $s_1=A_2.$
To deduce the existence of a fat inscribed triangle we aim  to apply Lemma \ref{lem:general box} to the line $\tilde l$. To do so it remains only to check that 
\begin{equation}\label{eq:trapezioid inside}
     \mathcal T\coloneqq \{(x,y) \ : \ x \in (s_0,s_1), \, \tilde l(x)\le y \le c\}\subset \Omega
\end{equation}
for some $c\ge\tilde  l(s_i)+\frac13,$ $i=0,1.$
First observe that, since $2A\ge 6\ge |s_0-s_1|$, at least one of the following holds
\[
[s_0,s_1]\subset [-A_1,A],\quad [s_0,s_1]\subset  [-A,A_2].
\]
If the  first case  holds we have $\tilde l(s_i)\le -a$ and we take $c\coloneqq -\frac a2$, $i=0,1.$ If the first does not hold, the second does and we have $\tilde l(s_i)\le 0$, $i=0,1,$ so we can take $c\coloneqq \frac a2$. With this choice of $c$ we also have $[s_0,s_1]\times(-\infty,c]\subset \mathcal H$ (recall \eqref{eq:def H}). To show \eqref{eq:trapezioid inside} we can now argue exactly as in the proof of Lemma \ref{lem:easy case}.  Suppose by contradiction that there is a point $R=(x_R,y_R)  \in \mathcal T\setminus  \Omega$, i.e.\  with $x_R\in(s_0,s_1)$ and $y_R \in  [\tilde l(x_R),c]$. 
Then the vertical segment $\{x_R\}\times [\tilde l(x_R),c]$ intersects $\partial \Omega$, because it has one endpoint in $R\in \rr^2\setminus \cl(\Omega)$ and one endpoint in $(x_R,c)\in \Omega$, by our choice of $c$.  Moreover, as observed above, $\{x_R\}\times (-\infty,c]\subset \mathcal H$, therefore $\{x_R\}\times [\tilde l(x_R),c]$ intersects $C$. However this is in contradiction with \eqref{eq:C below tilde l}.
This proves \eqref{eq:trapezioid inside} and so we can apply Lemma \ref{lem:general box} to deduce that $\Omega$ has an inscribed triangle $\T$ with $\alpha(\T)\ge \arctan\left(\frac1{47}\right)$ and $A(\T)\ge \frac{9}{4420}.$

We are left with the following last case.

\noindent {\color{blue} \sc Case 3b}: $(T_m\cup T_l)\cap \partial \Omega=\emptyset$, where the (closed) triangles $T_m,T_l$ were defined immediately before \textsc{Case 3.a} (see also Figure \ref{fig:case3}). This assumption is equivalent to 
\begin{equation}\label{eq:final}
(T_m\cup T_l)\subset \Omega,
\end{equation}
indeed both $T_m$ and $T_l$ intersect $\Omega$ (because they both intersect the interior of the bottom and top horizontal sides of $\mathcal R$ which are contained in $\Omega$ by construction). Observe that, performing a rotation of $\pi$-degrees around the origin,  \eqref{eq:final} remains true, but the values of $A_1$ and $A_2$ are swapped (see also Figure \ref{fig:case3}).  In particular, without loss of generality, we can assume from now on that 
\begin{equation}\label{eq:A2big}
    A_2\ge A_1.
\end{equation}
The assumption \eqref{eq:final}  implies that:
\begin{equation}\label{eq:Q contained}
    \parbox{12cm}{the closed quadrilateral $\mathcal{Q}$ determined by the lines $l,m$ and the vertical lines $v_1\coloneqq \{-A\}\times \rr$, $v_2\coloneqq \{A\}\times \rr$, is contained in $\cl(\Omega)$ (see Figure \ref{fig:case3.b}). Moreover $\mathcal Q\cap \partial \Omega$ is contained in the vertical lines $v_1,v_2.$}
\end{equation}
Indeed $\mathcal Q$ it is contained in $\mathcal R\cup T_m\cup T_l \subset \cl(\Omega)$, by \eqref{eq:final}, and $\mathcal R$ intersects $\partial\Omega$ only on $\{A\}\times \rr$ and $\{-A\}\times \rr$ (by \eqref{eq:R new assumptions}), while $T_m,T_l$ do not intersect $\partial \Omega.$
Moreover, since  we are assuming that $A_2\ge A_1$,  we have that the slope of the line $l$ is not greater than the slope of $m$ (see \eqref{eq:equations l and m} for the expression of $m$ and $l$). In particular the line   $y=\tilde m(x)$ that is parallel to $l$ and such that $\tilde m(-A)=m(-A)$ satisfies $\tilde m(x)\le m(x)$ for every $x\ge -A$. Therefore by \eqref{eq:Q contained}
\begin{equation}\label{eq:P contained}
    \parbox{12cm}{the closed parallelogram $\mathcal P$ determined by $l, \tilde m$  and the vertical lines $v_1,v_2$, is contained in $\cl(\Omega)$ and $\mathcal P\cap \partial \Omega$ is contained in the vertical lines $v_1,v_2$ (see Figure \ref{fig:case3.b}).}
\end{equation}
Indeed $\mathcal P\subset \mathcal Q$.
 The explicit expression for $\tilde m$ can be easily computed to be:
\begin{equation}\label{eq:equation tilde m}
    \tilde m(x)\coloneqq (x+A)\frac{a}{A_2-A+3}+(A_1-A)\frac{a}{A_1-A+3}.
\end{equation}
This formula (and the one for $m$ in \eqref{eq:equations l and m}) also provides an alternative direct way to show that $\tilde m(x)\le m(x)$ for $x\ge -A,$ provided $A_2\ge A_1.$

We are now ready to build the closed rectangle $\tilde {\mathcal R}\subset \Omega$ required by point $ii)$ of the statement. We will actually construct a closed rectangle $\tilde {\mathcal R_0}\subset \cl(\Omega)$, but which intersects $\partial \Omega$ only at the vertices, from which $\tilde {\mathcal R}\subset \Omega $ will be obtained by slightly shrinking $\tilde {\mathcal R_0}$ in the direction of the longer side. 
Let $P_2$ be the point of intersection between $l$ and $v_2$ and $P_1$ be the point of intersection between $\tilde m$ and $v_1$. Let also $Q_1,Q_2$ be the orthogonal projections of $P_1,P_2$ respectively on $l$ and $\tilde m$ (see Figure \ref{fig:case3.b}). We define $\tilde {\mathcal R_0}$ as the closed rectangle of vertices  $P_1,Q_1,P_2,Q_2$. Observe that in Figure \ref{fig:case3.b} the points $Q_1,Q_2$ are inside the strip $(-A,A)\times \rr$, however this requires some justification (see comment after \eqref{eq:a-a upper}), indeed a priori if the lines $\tilde m$ and $l$ are too vertical, it might happen that $Q_1,Q_2$ fall outside this strip. 
To conclude the proof it is sufficient to prove that $\tilde {\mathcal{R}_0}\subset \cl( \Omega)$, that $\tilde {\mathcal{R}_0}\cap \partial \Omega\subset \{P_1,P_2\}$ and that $\tilde {\mathcal{R}_0}$ has sides-length $2\tilde a \times 2\tilde A$ satisfying:
	\begin{equation}\label{eq:key ineq proof}
		\begin{split}
				&  \tilde A-  A>2\sqrt 2(a-\tilde a),\quad 2\tilde a-2a \ge \frac{1}{3\diam(\Omega)} ,\quad \tilde a \le 1.
		\end{split}
	\end{equation}
 Indeed we can then slightly  shrink the $\tilde {\mathcal{R}_0}
 $ along the direction of side of length $2\tilde A$ (keeping the width $2\tilde a$ intact) and obtain a closed rectangle such that $\tilde {\mathcal R}\subset \Omega$  and satisfying $ii)$ of the statement. Indeed if the length of the longer side of $\mathcal R$ is sufficiently close to $2\tilde A$, then the requirement \eqref{eq:Acontrol},  present in $ii)$ of the statement, will be satisfied. Note that the inclusion $\tilde {\mathcal R}\subset \Omega$ is ensured by the fact that $\tilde {\mathcal{R}_0}\cap \partial \Omega\subset \{P_1,P_2\}$.
To show the required properties of $\tilde {\mathcal{R}_0}
 $ we need first to obtain some inequalities relating the quantities $A,A_2$ and $A_1:$
\begin{equation}\label{eq:superbad}
	\begin{split}
			&A_1\le A_2\le 2A_1, \quad A_1,A_2\ge \frac32 A.
	\end{split}
\end{equation}
We start with the first one, where we only need to show that $2A_1\ge A_2,$ since we are assuming $A_2\ge A_1$ (recall \eqref{eq:A2big}). The idea is that, if  $A_2$ is much bigger than $A_1$, then the triangle $T_m$ is too wide and (since is contained in $\cl( \Omega)$) this would contradict $r(\Omega)=1$. More precisely we must have that  $m(A_2-3)\le 2,$ otherwise the rectangle $[A_2-3,A_2]\times[0,m(A_2-3)]$, which is contained in $\cl( \Omega)$, has both sides strictly grater than 2 (see Figure \ref{fig:case3bis}). Therefore recalling the expression of $m$ in \eqref{eq:equations l and m} we get
\[
2\ge m(A_2-3)=\frac{a(A_1+A_2-3)}{A_1-A+3}\ge \frac23\frac{(A_1+A_2-3)}{A_1-1},
\]
where we used that by hypotheses $a\ge \frac23$ and $A\ge 4.$ Rearranging the terms this shows the first in \eqref{eq:superbad}. For the second in \eqref{eq:superbad} the idea is similar.
Indeed we must have that  $m(A-3)\le 2-a,$   otherwise the rectangle  $[A-3,A]\times[-a,m(A-3)]$, which  is contained in $\cl( \Omega)$, would have  both sides strictly grater than 2 (see Figure \ref{fig:case3bis}). Therefore using again the equation for $m$ we get
\[
2-a\ge m(A-3)=\frac{a(A_1+A-3)}{A_1-A+3},
\]
and simplifying 
\[
A\le (1-a)A_1+3\le \frac13 A_1+3 \le \frac13 A_1+\frac A2 .
\]
where we used that by hypotheses $a\ge \frac23$ and $A\ge 6.$ This 
shows that $A_1 \ge \frac32 A$  and thus the second in \eqref{eq:superbad} (recall that $A_2\ge A_1$).
 For convenience we now denote  
 $$\eta\coloneqq \frac{a}{A_2-A+3}$$ the slope of the parallel lines $l$ and $\tilde m$, so that their respective expression becomes:
\begin{equation}\label{eq:equation  l m eta}
    l(x)=\eta x-A_2\eta,\quad \tilde m(x)=\eta x+A\eta+(A_1-A)\frac{a}{A_1-A+3}
\end{equation}
(recall \eqref{eq:equations l and m} and \eqref{eq:equation tilde m}).
The following inequalities will play a key role
\begin{equation}\label{eq:bound eta}
\frac{2}{3\diam(\Omega)}\le \frac{a}{A_2}\le  \eta\le \frac{3a}{A_2}\le \frac3{A_2}< 1,	
\end{equation}
where in the first inequality we used that $A_2\le \diam(\Omega)$   and $a\ge 2/3$, in the second that $A\ge 3$, while in the third one we used  $A\le \frac23 A_2$ by \eqref{eq:superbad},  in the fourth step we used that $a\le 1$  and finally in the last step that $A_2\ge A> 3.$ In particular the slope $\eta$ is comparable to $\frac a{A_2}.$
\begin{figure}[!ht]
\centering
\includegraphics[width=14.5cm, height=7cm]{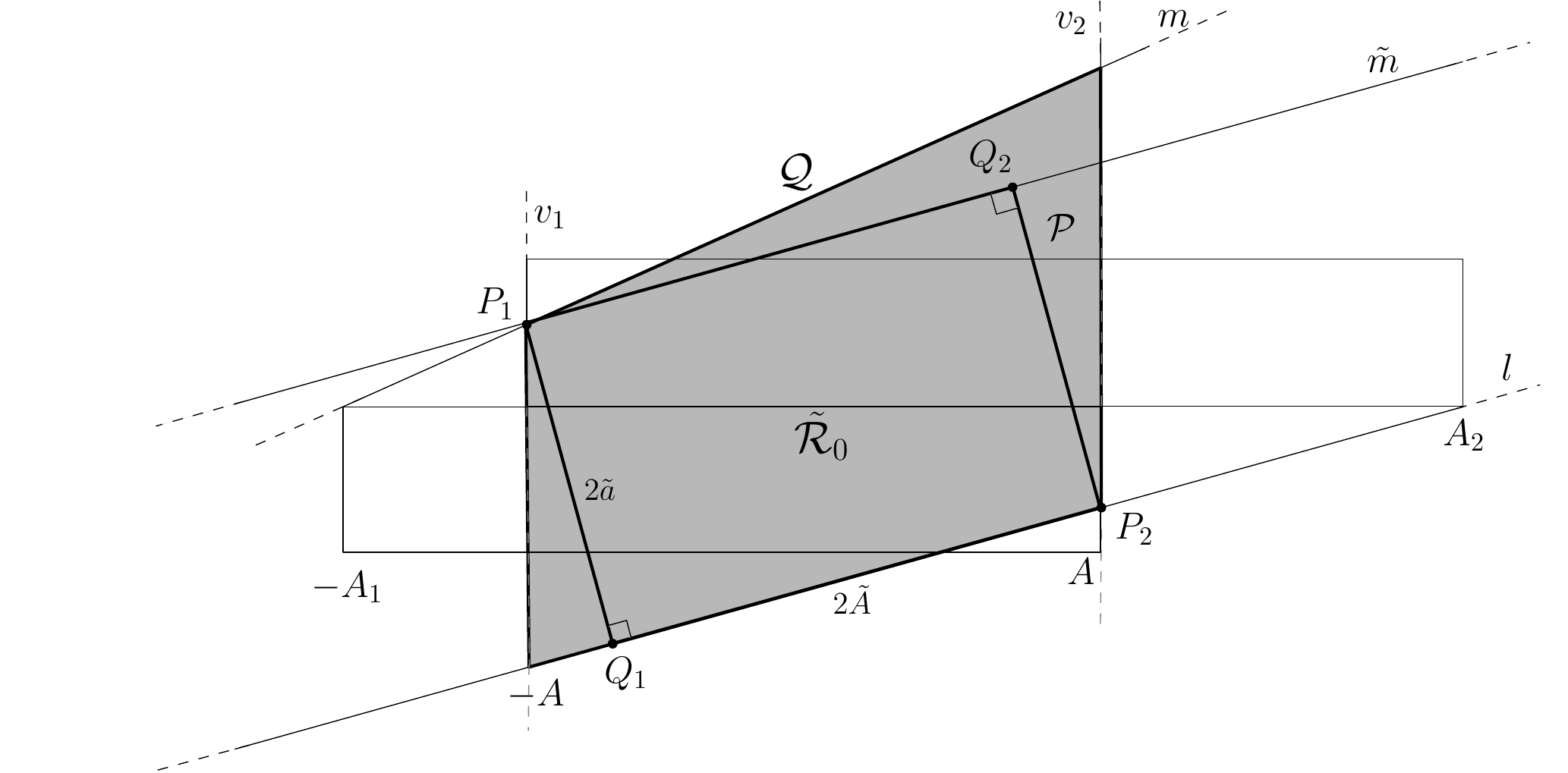}
\caption{The construction of the rectangle $\tilde {\mathcal R}_0$ of vertices $P_1,P_2,Q_1,Q_2$ in  \textsc{Case 3.b}. With $\mathcal Q$ and $\mathcal P$ are denoted the quadrilaterals enclosed  by the lines $v_1,v_2,$ $l$ and respectively $m$ or $\tilde m$}
\label{fig:case3.b}
\end{figure}

We can finally begin the verification of \eqref{eq:key ineq proof}.  Denote by $2\tilde a$ the length of the sides ${P_2Q_2}$ and $P_1Q_1$ and by $2\tilde A$ the length of the sides $Q_1P_2$, $P_1Q_2$. Then $2\tilde a$ is equal to the distance between the parallel lines $l$ and $\tilde m$ and is given by the following formula
\begin{equation}\label{eq:bound a}
	2\tilde a\coloneqq \frac{|\tilde m(0)-l(0)|}{\sqrt{1+\eta^2}}\overset{\eqref{eq:equation  l m eta}}{=}\frac{\eta (A_2+A)+(A_1-A)\frac{a}{A_1-A+3}}{\sqrt{1+\eta^2}}.
\end{equation}
Using  the expression of $\eta$ we can compute
\begin{equation}\label{eq:ez}
	\begin{split}
		2\tilde a-2a&= a\frac{\frac{A_2+A}{A_2-A+3}+\frac{A_1-A}{A_1-A+3}}{\sqrt{1+\eta^2}}-2a=\frac a{\sqrt{1+\eta^2}} \left(\frac{2A-3}{A_2-A+3}-\frac{3}{A_1-A+3}-2\sqrt{1+\eta^2}+2\right)\\
		&\ge \frac a{\sqrt{2}}\left(\frac{2A-3}{A_2-A+3}-\frac{3}{A_1-A}-\eta\right)\overset{ A\le 2/3A_1}{\ge} \frac{a}{\sqrt{2}}\left(\frac{2A-3}{A_2-A+3}-\frac{3}{A_1/3}-\eta\right)\\
  &\overset{ 2A_1\ge A_2}{\ge}\frac{a}{\sqrt{2}}\left(\frac{2A-3}{A_2-A+3}-\frac{18}{A_2}-\eta\right)\ge \frac{1}{\sqrt{2}}\left(\eta (2A-3)-18\eta-\eta\right)\\
  &\ge \frac{2A-22}{\sqrt{2}}\eta\ge 2\sqrt2 \eta,
	\end{split}
\end{equation}
where in the first inequality we used that $\sqrt{1+\eta^2}\le 1+\frac12 \eta $ and that $\eta \le 1$ (by \eqref{eq:bound eta}), in the fourth inequality we used that $\frac{a}{A_2}\le \eta $ (by \eqref{eq:bound eta}) and in the last step the assumption $A\ge 13$. Using the first in \eqref{eq:bound eta} we reach
\[
2\tilde a-2a\ge \frac{4\sqrt 2}{3\diam(\Omega)}\ge \frac{1}{3\diam(\Omega)},
\]
that is the second in \eqref{eq:key ineq proof}.
It is also useful to derive the following upper bound
\begin{equation}\label{eq:a-a upper}
    2\tilde a-2a=\frac{\eta (A_2+A)+a\frac{A_1-A}{A_1-A+3}}{\sqrt{1+\eta^2}}-2a
    \le 6+a-2a\le 6-\frac23
\end{equation}
where in the first inequality we used that $\eta\le 3/A_2$ (by \eqref{eq:bound eta}). In particular $2\tilde a\le 8<2A$, since by assumption $A> 4$ and $a\le 1$. This implies that $Q_1,Q_2$ are contained in the strip   $(-A,A)\times \rr$ and so the rectangle $\tilde{\mathcal R_0}$ is contained in the parallelogram $\mathcal P$ and thus is contained in $\cl( \Omega)$ by \eqref{eq:P contained}(see Figure \ref{fig:case3.b}).  This also shows that $\tilde{\mathcal R_0}\cap \partial \Omega \subset \{P_1,P_2\}$, indeed $\mathcal P\cap \partial \Omega $ is contained in the vertical lines $v_1,v_2$ (again by \eqref{eq:P contained}) and by construction $\tilde{\mathcal R_0}$ intersects these lines only at the vertices $P_1,P_2.$
 We pass to the proof of the first in \eqref{eq:key ineq proof}. 
The value of $2\tilde A= \overline{Q_1P_2}=\overline{Q_2P_1}$ can be computed using the Pythagorean theorem: 
\[
2\tilde A=\sqrt{\overline{P_1P_2}^2-\overline{P_1Q_1}^2}=\sqrt{\overline{P_1P_2}^2-(2\tilde a)^2}
\]
(see Figure \ref{fig:case3.b}).
Moreover, since $P_1=(-A,\tilde m(-A))$ and $P_2=(A,l(A)),$ we have
\begin{align*}
    \overline{P_1P_2}^2&=(2A)^2+\left|\tilde m(-A)-l(A)\right|^2=(2A)^2+\left|\frac{a(A_1-A)}{A_1-A+3}-\eta(A-A_2)\right|^2
\end{align*}
(recall \eqref{eq:equation  l m eta}).
Combining the last two identities and the expression for $2\tilde a$ given in \eqref{eq:bound a}
\begin{align*}
	(2\tilde A)^2&=\overline{P_1P_2}^2-(2\tilde a)^2\\
 &\ge (2A)^2+\left|\frac{a(A_1-A)}{A_1-A+3}-\eta(A-A_2)\right|^2-\left(\frac{a(A_1-A)}{A_1-A+3}+\eta (A+A_2)\right)^2\\
	&= (2A)^2-4\eta A\left(\eta A_2+\frac{a(A_1-A)}{A_1-A+3}\right)\ge (2A)^2-4\eta A\left(\eta A_2+1\right)\\
 &\ge  (2A)^2-16\eta A=(2A)^2 \left(1-\frac{4}{A}\eta\right),
\end{align*}
where in the last inequality we used $\eta\le \frac3{A_2}$ (see \eqref{eq:bound eta}).
Since $\eta\le 1$ (by \eqref{eq:bound eta}) and by assumption $A\ge 4,$ we have $\frac{4}{A}\eta\le 1$.
 Hence we can take the square root on both sides and obtain
\[
2\tilde A\ge 2A\sqrt{1-\frac{4}{A}\eta}> 2A- 8\eta,
\]
where we used
$\sqrt{1-t}>1-t$ for all $t\in(0,1)$. Combing this with \eqref{eq:ez} we obtain that
\begin{equation}\label{eq:key inequality end}
    \tilde A- A>{2\sqrt 2}( a-\tilde a),
\end{equation}
which is the first in \eqref{eq:key ineq proof}. It remains only to show that $\tilde a \le 1$ (i.e.\ the third in \eqref{eq:key ineq proof}). 
Combining \eqref{eq:key inequality end} and \eqref{eq:a-a upper} we have 
$$\tilde A\ge A-  {2\sqrt 2}( a-\tilde a)> 13 -{2\sqrt 2}\big(3-\frac13\big)>1,$$ 
hence we must have that $\tilde  a\le 1$, otherwise the rectangle $\tilde {\mathcal R_0}$ would have both sides of length strictly larger than $2$, which contradicts $r(\Omega)=1.$
\end{proof}

\section{Proof of the technical Lemma \ref{lem:trick set}}\label{sec:lemma}
Here we prove Lemma \ref{lem:trick set}, which was the main technical tool for the proof Theorem \ref{thm:main} and more precisely for the intermediate Proposition \ref{prop:cases}. We decided to isolate here its proof, because the result is essentially equivalent to the following independent lemma about upper semicontinuous functions in the unit interval.  Lemma \ref{lem:trick set} will be restated and proved  immediately after.
\begin{lemma}\label{lem:trick}
	Let $f:[0,1]\to \rr$ be an upper semicontinuous function and let $l:\rr\to \rr$ be  an affine function such that $l(0)\le f(0),$ $l(0)\le f(1).$ Suppose that
 \begin{equation}\label{eq:below}
     f(t)\le l(t), \quad  \forall \, t \in(t_0,t_1),
 \end{equation}
	for some $t_0,t_1\in [0,1]$, $t_0\le t_1.$ 
	Then for every $\lambda\le |t_0-t_1|$ there exist $s_0,s_1 \in[0,1]$ with $s_0<s_1$, $|s_0-s_1|\in [\lambda,2\lambda]$ and such that
 \begin{equation}\label{eq:f le l}
     f(t)\le \tilde l(t), \quad  \forall\, t \in(s_0,s_1)
 \end{equation}
	where $\tilde l$ is the affine function such that $\tilde l(s_i)=f(s_i),$ $i=0,1.$ Moreover if \eqref{eq:below} holds with strict inequality then $s_0,s_1$ can be chosen so that \eqref{eq:f le l} is strict as well.
\end{lemma}
\begin{proof}
 It is enough to prove that for every $\eps>0$ there exist $s_0,s_1 \in[0,1]$ with $s_0<s_1$, $|s_0-s_1|\in [\lambda-\eps,2\lambda]$ and such that \eqref{eq:f le l} holds (with strict inequality if \eqref{eq:below} is also strict). Indeed if this is true then we can find  a sequence of points $s_0^n<s_1^n$ such that \eqref{eq:f le l} holds for $s_0^n,s_1^n$ (with strict inequality if \eqref{eq:below} is also strict)  and $s_0^n\to s_0$ and $s_1^n\to s_1$ for some  $s_0,s_1 \in[0,1]$ with $s_0<s_1$, $|s_0-s_1|\in [\lambda,2\lambda]$. Then \eqref{eq:f le l}  (for $s_0^n,s_1^n$)  can be rewritten as 
\begin{equation}\label{eq:flel2}
    f(t)\le \frac{(s_1^n-t)f(s_0^n)+(t-s_0^n)f(s_1^n)}{s_1^n-s_0^n}, \quad \forall t \in (s_0^n,s_1^n).
\end{equation}
In particular, since $f$ is upper-semicontinuous, taking the limsup on the right hand side we obtain that \eqref{eq:f le l} holds also  for $s_0,s_1.$ Moreover if \eqref{eq:below} is strict we claim that we can choose $s_0,s_1$ so that \eqref{eq:f le l} also holds with strict inequality. If $|s_0-s_1|\in(\lambda,2\lambda]$ we can just take $s_0=s_0^n,s_1=s_0^n$ for $n$ big enough, hence we assume that $|s_0-s_1|=\lambda$. We also assume from now on that \eqref{eq:f le l} holds with equality for some $t \in (s_0,s_1) $, i.e.\ that $f(t)=\tilde l(t)$, where  $\tilde l$ is the affine function such that $\tilde l(s_i)=f(s_i),$ $i=0,1.$ If this was not the case we would be done. Suppose now that exists a subsequence for which  $|s_1^n-s_0^n|\ge |s_0-s_1|$ holds.  Then we can just replace $s_0,s_1$ with $s_0^n,s_1^n$ for $n$ big enough (for which in particular \eqref{eq:f le l} holds with strict inequality). Hence suppose that $|s_1^n-s_0^n|< |s_0-s_1|$ for every $n$ big enough, i.e.\ for every $n$ big enough either $s_1^n<s_1$ or  $s_0^n>s_0$ holds. Up to symmetry and up to passing to a subsequence we can assume that $s_0^n>s_0$ holds for every $n$. In particular, since $s_0^n\in (s_0,s_1)$ for $n$ big enough, by \eqref{eq:f le l} for $s_0,s_1$ we have that $f(s_0^n)\le \tilde l(s_0^n)$.  Moreover for $n$ big enough we have   $s_0<s_0^n<t<s_1^n$, where $t$ is as above. Hence we must have that $f(s_1^n)>\tilde l(s_1^n),$ otherwise, since \eqref{eq:f le l} for $s_0^n,s_0^n$ holds with strict inequality, we would have  $f(t)<\tilde l(t).$   Again by \eqref{eq:f le l} (for $s_0,s_1$) this also shows  that $s_1^n\ge s_1.$ Hence taking $s_0'=s_0$ and $s_1'=s_1^n$ we obtain that \eqref{eq:f le l} holds for $s_0'$ and $s_1'$ with strict inequality (recall that $\tilde l(s_0^n)\ge f(s_0^n)$). Moreover for $n$ big enough we have $|s_0'-s_1'|\in [\lambda,2\lambda],$ indeed $|s_0'-s_1'|\ge |s_0-s_1|=\lambda$ and $s_1^n\to s_1.$ 

Summarizing it remains to prove that:
\begin{equation}\label{eq:it remains}
    \parbox{13cm}{for every $\lambda \le|t_0-t_1|$ and  every $\eps>0$ there exist $s_0,s_1 \in[0,1]$ with $s_0<s_1$, $|s_0-s_1|\in [\lambda-\eps,2\lambda]$ and such that \eqref{eq:f le l} holds (with strict inequality if \eqref{eq:below} is also strict).}
\end{equation}
We prove statement \eqref{eq:it remains} in two steps.

{\noindent {\bf Step 1}:} $t_0=0$, $t_1=1.$  
If $\lambda\ge \frac12$ we can simply take $s_0=0$ and $s_1=1$. Indeed the affine function $\tilde l$ that agrees with $f$ at $0$ and $1$ satisfies $\tilde l\ge l$  in $[0,1]$. Hence by \eqref{eq:below} we have that \eqref{eq:f le l} holds (with strict inequality if \eqref{eq:below} is also strict). Hence we only need to prove \eqref{eq:it remains} for $\lambda\le \frac12.$ 
We claim that to do this it is enough to  show that:
\begin{equation}\label{eq:step 1 trick}
    \parbox{13cm}{for every $\eps\in(0,\frac12)$ there exist $s_0,s_1\in [0,1]$ with $s_0<s_1$, $|s_0-s_1|\in [\frac{1}{2}-\eps,1-\eps]$ and satisfying \eqref{eq:f le l}  (with strict inequality if \eqref{eq:below} is strict).}
\end{equation}
To see this observe that every $s_0,s_1$ that satisfy \eqref{eq:f le l} also satisfy the hypotheses of the statement (up to scaling and translating $[s_0,s_1]$ to $[0,1]$) with $t_0=s_0$ and $t_1=s_0$, i.e.\ under the additional assumptions of the current Step 1.
Therefore, setting $s_0^0\coloneqq 0,$ $s_1^0\coloneqq 1$  and iterating \eqref{eq:step 1 trick}, for every  $n \in \nn$  we can find  $s_0^n,s_1^n \in[0,1]$ such that
 $|s_0^n-s_1^n|\in ((\frac{1}{2}-\eps)|s_0^{n-1}-s_1^{n-1}|,(1-\eps)^n)$ and \eqref{eq:f le l} holds with $s_0=s_0^n$ and $s_1=s_1^n$  (with strict inequality if \eqref{eq:below} is strict). In particular $|s_0^n-s_1^n|\to 0.$ Then for any $\lambda\le \frac12 $ we can choose the first $n\in \nn $ such that $|s_0^n-s_1^n|<2\lambda$, in particular $|s_0^{n-1}-s_1^{n-1}|\ge 2\lambda$ (this is true even if $n=1$, since $|s_0^0-s_1^0|=1$ and $\lambda \le \frac12$). Hence 
 $$|s_0^n-s_1^n|\ge (\frac{1}{2}-\eps) |s_0^{n-1}-s_1^{n-1}| \ge (\frac{1}{2}-\eps)2\lambda.$$
 This would prove \eqref{eq:it remains} and would conclude the proof in the Step 1 case.

We pass to the proof of \eqref{eq:step 1 trick}.
It  is enough to consider the case $f(0)=l(0)$, $f(1)=l(1)$, otherwise we can replace $l$ with the affine function agreeing with $f$ at 0 and 1, which still satisfies the hypotheses, being bigger than $l$ in $[0,1].$ Up to adding an affine  function to both $f$ and $l$, we can further assume that $f(0)=f(1)=0$. In particular $f\le 0$ in $(0,1)$ (with strict inequality if \eqref{eq:below} is strict).Without loss of generality we can assume that 
	\begin{equation}\label{eq:ass}
			M\coloneqq \max_{[1/2,1-\eps]} f\ge 	\max_{[\eps,1/2]} f.
	\end{equation}
	 Note that both maximum exist by upper-semicontinuity and that $M\le 0$  (with strict inequality if \eqref{eq:below} is strict). Let $s_1 \in [1/2,1-\eps]$ be such that $f(s_1)=M$. Consider the affine function $r$ such that $r(s_1)=f(s_1)=M$ and $r(0)=0=f(0).$ Set
	 \[
	 s_0=\max \{t \in [0,\eps]\ : \ f(t)\ge r(t) \},
	 \]
	 which exists again by upper-semicontinuity. Clearly $|s_0-s_1|\in [\frac{1}{2}-\eps,1-\eps]$. Moreover by \eqref{eq:ass} and since $r\ge M$ in $(-\infty,s_1)$ (with strict inequality if \eqref{eq:below} is strict) we have
	 $$
	 f(t)\le r(t), \quad t \in (s_0,s_1)
	 $$
  (with strict inequality if \eqref{eq:below} is strict).
	The conclusion follows noting that, since $f(s_i)\ge r(s_i)$, $i=0,1,$ then $r\le \tilde l$ in  $(-\infty,s_1]$, where $\tilde l$ is the affine function such that $\tilde l(s_i)=f(s_i),$ $i=0,1.$
	
	 	{\noindent {\bf Step 2}:} $t_0,t_1$ arbitrary.  Fix any $\lambda\ge 0$ such that $\lambda \le |t_0-t_1|$ and  any $\eps>0$. Let  $l$ be  as in the statement. Set
	 	\[
	 	t_0'\coloneqq\max\{s \in[0,t_0]\ : \ f(s)\ge l(s)\}, 	\quad 	t_1'\coloneqq\min\{s \in[t_1,1] \ : \ f(s)\ge l(s)\}.
	 	\]
	 	Both maximum and minimum exist finite because $f$ is upper-semicontinuous and by assumption $f(0)\ge l(0)$ and $f(1)\ge l(1).$ Clearly $[t_0,t_1]\subset [t_0',t_1'].$  Then $f	\le l$ in $(t_0',t_1')$, with strict inequality if \eqref{eq:below} was strict. Denoting by $l'$  the affine function such that $ l'(t_i')=f(t_i'),$ $i=0,1,$ we have $f\le l\le l'$ in $(t_0',t_1')$, indeed $l'(t_i')=f(t_i')\ge l(t_i')$, $i=0,1$. Therefore, up to  rescaling the interval $[t_0',t_1']$ to $[0,1]$, we can apply \eqref{eq:it remains} (in the case $t_0=0,t_1=1$, which is true by  Step 1) 
   and deduce that for all  $ \lambda' \le 1$ 
 there exist $s_0,s_1 \in [t_0',t_1']$ such that 
 $$|s_0-s_1|\in [(\lambda'-\eps)|t_0'-t_1'|, 2\lambda'|t_0'-t_1'|]$$
 and \eqref{eq:f le l} holds (with strict inequality if \eqref{eq:below} was strict). In particular $|s_0-s_1|\in [\lambda'|t_0'-t_1'|-\eps, 2\lambda'|t_0'-t_1'|]$, since $|t_0'-t_1'|\le 1$. Taking  $\lambda'=\frac\lambda {|t_0'-t_1'|}\le \frac{|t_0-t_1|}{ |t_0'-t_1'|}\le  1$ proves \eqref{eq:it remains}. This concludes the proof.
\end{proof}
With the above result the key Lemma \ref{lem:trick set} now follows almost immediately. We restate it below for the convenience of the reader,
\begin{lemma}[Two points above-interval below]
    Let $C\subset \rr^2$ be a closed  set such that $C\subset \{(x,y)\ : \ y\le a\}$ for some $a \in \rr$ and let  $y=l(x)$ be a (non-vertical) line.   Suppose that:
    \begin{enumerate}[label=\roman*)]
        \item $C$ contains two points $(x_i,y_i)$, $i=0,1$, such that $y_i\ge l(x_i)$,  $i=0,1,$
        \item $C\cap \{x=t\}\neq\emptyset$ for all $t\in[x_0,x_1]$,
        \item there exists $[t_0,t_1]\subset [x_0,x_1]$ such that $C\cap\{t_0<x<t_1\}\subset l_\leq\coloneqq \{(x,y) \ : \ y\le l(x)\}$.
    \end{enumerate}
Then for every $\lambda\le |t_0-t_1|$ there exists $[s_0,s_1]\subset [x_0,x_1] $ with  $|s_0-s_1|\in [\lambda,2\lambda]$ and a line $y=\tilde l(x)$ such that  $C\cap\{s_0<x<s_1\}\subset \tilde l_\leq$ for every $t\in(s_0,s_1)$ and such that $(s_i,l(s_i))\in C$, $i=0,1.$ Moreover if the inequality in \ref{it:below} is strict, i.e.\  $C\cap\{t_0<x<t_1\}\subset l_<$,  we can also choose $t_0,t_1$ so that $C\cap\{s_0<x<s_1\}\subset \tilde l_<$.
\end{lemma} 
\begin{proof}
Define  $f:[x_0,x_1]\to \rr$ by
    \[
	f(x)\coloneqq \max \{y \ : (x,y) \in C\}. 
	\]
The maximum exists finite for every $x \in[x_0,x_1]$ because of the assumption  $C\subset \{(x,y)\ : \ y\le a\}$ and since $C$ is closed and intersects all vertical lines $\{x=t\}$ for all $t \in [x_0,x_1]$. In particular $(x,f(x))\in C$ for every $x \in [x_0,x_1].$   We claim that $f$ is upper-semicontinuous. Indeed let $x \in [x_0,x_1]$ and $x_n \in (x_0,x_1)$ with $x_n \to x$. Take a subsequence $x_{n_k}$ so that $y\coloneqq \lim_k f(x_{n_k})=\limsup_n f(x_n).$ Then $(x_{n_k},f(x_{n_k}))\to (x,y)\in C$, by closeness, hence by definition $f(x)\ge y$, which proves the claim. Moreover we have $f(x_i)\ge l(x_i)$ for $i=0,1$, because by assumption the points $(x_i,y_i)$, $i=0,1,$ with $y_i\ge l(x_i)$, belong to $C.$ Finally $f\le l$ in $(t_0,t_1)$ by construction, with strict inequality if we have a strict inequality in $iii)$. Hence an application of Lemma \ref{lem:trick}, up to a rescaling, yields the result.
\end{proof}

\def\cprime{$'$} \def\cprime{$'$}


\begin{thebibliography}{10}
	
	\bibitem{comp3}
	{\sc P.~K. Agarwal, N.~Amenta, and M.~Sharir}, {\em Largest placement of one
		convex polygon inside another}, Discrete Comput. Geom., 19 (1998),
	pp.~95--104.
	
	\bibitem{cyclic}
	{\sc A.~Akopyan and S.~Avvakumov}, {\em Any cyclic quadrilateral can be
		inscribed in any closed convex smooth curve}, Forum Math. Sigma, 6 (2018),
	pp.~Paper No. e7, 9.
	
	\bibitem{neck}
	{\sc J.~Aslam, S.~Chen, F.~Frick, S.~Saloff-Coste, L.~Setiabrata, and
		H.~Thomas}, {\em Splitting loops and necklaces: variants of the square peg
		problem}, Forum Math. Sigma, 8 (2020), pp.~Paper No. e5, 16.
	
	\bibitem{tria1}
	{\sc B.~S. Baker, E.~Grosse, and C.~S. Rafferty}, {\em Nonobtuse triangulation
		of polygons}, Discrete Comput. Geom., 3 (1988), pp.~147--168.
	
	\bibitem{tab1}
	{\sc B.~Baritompa, R.~L\"{o}wen, B.~Polster, and M.~Ross}, {\em Mathematical
		table-turning revisited}, Math. Intelligencer, 29 (2007), pp.~49--58.
	
	\bibitem{bruckner}
	{\sc A.~Bruckner}, {\em A {N}ote on {C}onvex {P}olygons {I}nscribed in {O}pen
		{S}ets}, Math. Mag., 37 (1964), pp.~250--251.
	
	\bibitem{tab4}
	{\sc R.~Fenn}, {\em The table theorem}, Bull. London Math. Soc., 2 (1970),
	pp.~73--76.
	
	\bibitem{rectangles}
	{\sc J.~E. Greene and A.~Lobb}, {\em The rectangular peg problem}, Ann. of
	Math. (2), 194 (2021), pp.~509--517.
	
	\bibitem{conv3}
	{\sc C.~M. Hebbert}, {\em The inscribed and circumscribed squares of a
		quadrilateral and their significance in kinematic geometry}, Ann. of Math.
	(2), 16 (1914/15), pp.~38--42.
	
	\bibitem{magic}
	{\sc C.~O. Horgan and J.~G. Murphy}, {\em On an angle with magical properties},
	Notices Amer. Math. Soc., 69 (2022), pp.~22--25.
	
	\bibitem{root3}
	{\sc C.~Hugelmeyer}, {\em Every smooth jordan curve has an inscribed rectangle
		with aspect ratio equal to $\sqrt 3 $}, arXiv:1803.07417,  (2018).
	
	\bibitem{aspratio}
	\leavevmode\vrule height 2pt depth -1.6pt width 23pt, {\em Inscribed rectangles
		in a smooth {J}ordan curve attain at least one third of all aspect ratios},
	Ann. of Math. (2), 194 (2021), pp.~497--508.
	
	\bibitem{conv2}
	{\sc D.~Ismailescu, M.~J. Kim, and E.~Wang}, {\em On the maximum area of
		inscribed polygons}, Elem. Math., 76 (2021), pp.~45--61.
	
	\bibitem{conv1}
	{\sc J.~Jer\'{o}nimo~Castro}, {\em On convex curves which have many inscribed
		triangles of maximum area}, Amer. Math. Monthly, 122 (2015), pp.~967--971.
	
	\bibitem{tab2}
	{\sc E.~H. Kronheimer and P.~B. Kronheimer}, {\em The tripos problem}, J.
	London Math. Soc. (2), 24 (1981), pp.~182--192.
	
	\bibitem{conv4}
	{\sc M.~Lassak}, {\em On relatively equilateral polygons inscribed in a convex
		body}, Publ. Math. Debrecen, 65 (2004), pp.~133--148.
	
	\bibitem{conv5}
	\leavevmode\vrule height 2pt depth -1.6pt width 23pt, {\em Approximation of
		convex bodies by inscribed simplices of maximum volume}, Beitr. Algebra
	Geom., 52 (2011), pp.~389--394.
	
	\bibitem{comp2}
	{\sc S.~Lee, T.~Eom, and H.-K. Ahn}, {\em Largest triangles in a polygon},
	Comput. Geom., 98 (2021), pp.~Paper No. 101792, 18.
	
	\bibitem{tria2}
	{\sc H.~Maehara}, {\em Acute triangulations of polygons}, European J. Combin.,
	23 (2002), pp.~45--55.
	
	\bibitem{comp1}
	{\sc J.~Matou\v{s}ek, J.~Pach, M.~Sharir, S.~Sifrony, and E.~Welzl}, {\em Fat
		triangles determine linearly many holes}, SIAM J. Comput., 23 (1994),
	pp.~154--169.
	
	\bibitem{survey1}
	{\sc B.~Matschke}, {\em A survey on the square peg problem}, Notices Amer.
	Math. Soc., 61 (2014), pp.~346--352.
	
	\bibitem{mats}
	\leavevmode\vrule height 2pt depth -1.6pt width 23pt, {\em Quadrilaterals
		inscribed in convex curves}, Trans. Amer. Math. Soc., 374 (2021),
	pp.~5719--5738.
	
	\bibitem{meyerson}
	{\sc M.~D. Meyerson}, {\em Equilateral triangles and continuous curves}, Fund.
	Math., 110 (1980), pp.~1--9.
	
	\bibitem{tab3}
	\leavevmode\vrule height 2pt depth -1.6pt width 23pt, {\em Convexity and the
		table theorem}, Pacific J. Math., 97 (1981), pp.~167--169.
	
	\bibitem{tab5}
	\leavevmode\vrule height 2pt depth -1.6pt width 23pt, {\em Remarks on {F}enn's
		``the table theorem'' and {Z}aks' ``the chair theorem''}, Pacific J. Math.,
	110 (1984), pp.~167--169.
	
	\bibitem{nielsen}
	{\sc M.~J. Nielsen}, {\em Triangles inscribed in simple closed curves}, Geom.
	Dedicata, 43 (1992), pp.~291--297.
	
	\bibitem{notet}
	{\sc V.~H. Pettersson, H.~A. Tverberg, and P.~R.~J. \"{O}sterg\aa rd}, {\em A
		note on {T}oeplitz' conjecture}, Discrete Comput. Geom., 51 (2014),
	pp.~722--728.
	
	\bibitem{trich}
	{\sc R.~E. Schwartz}, {\em A trichotomy for rectangles inscribed in {J}ordan
		loops}, Geom. Dedicata, 208 (2020), pp.~177--196.
	
	\bibitem{survey2}
	\leavevmode\vrule height 2pt depth -1.6pt width 23pt, {\em Rectangles, curves,
		and {K}lein bottles}, Bull. Amer. Math. Soc. (N.S.), 59 (2022), pp.~1--17.
	
	\bibitem{tao}
	{\sc T.~Tao}, {\em An integration approach to the {T}oeplitz square peg
		problem}, Forum Math. Sigma, 5 (2017), pp.~Paper No. e30, 63.
	
	\bibitem{peg}
	{\sc O.~Toeplitz}, {\em Equilateral triangles and continuous curves}, Über
	einige Aufgaben der Analysis situs, 4 (1911).
	
	\bibitem{comp4}
	{\sc I.~van~der Hoog, V.~Keikha, M.~L\"{o}ffler, A.~Mohades, and J.~Urhausen},
	{\em Maximum-area triangle in a convex polygon, revisited}, Inform. Process.
	Lett., 161 (2020), pp.~105943, 4.
	
	\bibitem{tria3}
	{\sc L.~Yuan}, {\em Acute triangulations of polygons}, Discrete Comput. Geom.,
	34 (2005), pp.~697--706.
	
\end{thebibliography}
\end{document}